\documentclass[11pt]{article}
\usepackage[tbtags]{amsmath}
\usepackage{amssymb}
\usepackage{amsthm}
\usepackage[misc]{ifsym}
\usepackage{cases}
\usepackage{mathrsfs}
\usepackage{graphicx}

\numberwithin{equation}{section}   
\normalsize

\title{\bf Direct Approach of Linear-Quadratic Stackelberg Mean Field Games of Backward-Forward Stochastic Systems \thanks{This work is supported by National Key R\&D Program of China (2022YFA1006104), National Natural Science Foundations of China (11971266, 12271304, 11831010), and Shandong Provincial Natural Science Foundations (ZR2022JQ01, ZR2020ZD24, ZR2019ZD42).}}

\author{\normalsize  Wenyu Cong\thanks{\it School of Mathematics, Shandong University, Jinan 250100, P.R. China, E-mail: congwenyu@mail.sdu.edu.cn} , Jingtao Shi\thanks{\it Corresponding author. School of Mathematics, Shandong University, Jinan 250100, P.R. China, E-mail: shijingtao@sdu.edu.cn}}

\newtheorem{mypro}{Proposition}[section]
\newtheorem{mythm}{Theorem}[section]
\newtheorem{mydef}{Definition}[section]

\begin{document}

    \maketitle

    \noindent{\bf Abstract:}\quad This paper is concerned with a linear-quadratic (LQ) Stackelberg mean field games of backward-forward stochastic systems, involving a backward leader and a substantial number of forward followers. The leader initiates by providing its strategy, and subsequently, each follower optimizes its individual cost. A direct approach is applied to solve this game. Initially, we address a mean field game problem, determining the optimal response of followers to the leader's strategy. Following the implementation of followers' strategies, the leader faces an optimal control problem driven by high-dimensional forward-backward stochastic differential equations (FBSDEs). Through the decoupling of the high-dimensional Hamiltonian system using mean field approximations, we formulate a set of decentralized strategies for all players, demonstrated to be an $(\epsilon_1, \epsilon_2)$-Stackelberg equilibrium.

    \vspace{2mm}

    \noindent{\bf Keywords:}\quad Stackelberg mean field game, direct approach, linear-quadratic stochastic optimal control, forward-backward stochastic differential equation, Stackelberg equilibrium

    \vspace{2mm}

    \noindent{\bf Mathematics Subject Classification:}\quad 93E20, 60H10, 49K45, 49N70, 91A23

    \section{Introduction}

    {\it Mean field games} (MFGs) have garnered increasing scholarly attention, finding applications in diverse fields such as system control, applied mathematics, and economics (\cite{Bensoussan-Frehse-Yam-13}, \cite{Gomes-Saude-14}, \cite{Caines-Huang-Malhame-17}, \cite{Carmona-Delarue-18}).
    MFG theory serves as a framework for describing the behavior of models characterized by large populations, where the influence of the overall population is significant, despite the negligible impact on individual entities.
    The methodological foundations of MFG, initially proposed by Lasry and Lions \cite{Lasry-Lions-07} and independently by Huang et al. \cite{Huang-Malhame-Caines-06}, have proven effective and tractable for analyzing weakly coupled stochastic controlled systems with mean field interactions, establishing approximate Nash equilibria.
    In particular, within the {\it linear-quadratic} (LQ) framework, MFGs offer a versatile modeling tool applicable to a myriad of practical problems. The solutions derived from LQ-MFGs exhibit noteworthy and elegant properties. Current scholarly discourse has extensively explored MFGs, particularly within the LQ framework (\cite{Li-Zhang-08}, \cite{Bensoussan-Sung-Yam-Yung-16}, \cite{Moon-Basar-17}, \cite{Huang-Zhou-20}).
    Huang et al. \cite{Huang-Caines-Malhame-07} conducted a study on $\epsilon$-Nash equilibrium strategies in the context of LQ-MFGs with discounted costs.
    This investigation was rooted in the {\it Nash certainty equivalence} (NCE) approach. Subsequently, the NCE approach was applied to scenarios involving long run average costs, in Li and Zhang \cite{Li-Zhang-08}.
    In the domain of MFGs featuring major players, Huang \cite{Huang-10} delved into continuous-time LQ games, providing insights into $\epsilon$-Nash equilibrium strategies.
    Huang et al. \cite{Huang-Wang-Wu-16} introduced a backward-major and forward-minor setup for an LQ-MFG, and decentralized $\epsilon$-Nash equilibrium strategies for major and minor agents were obtained.
    Huang et al. \cite{Huang-Wang-Wu-16b} examined the backward LQ-MFG of weakly coupled stochastic large population systems under both full and partial information scenarios.
    Huang and Li \cite{Huang-Li-18} delved into an LQ-MFG concerning a class of stochastic delayed systems.
    Xu and Zhang \cite{Xu-Zhang-20} explored a general LQ-MFG for stochastic large population systems, where the individual diffusion coefficient is contingent on the state and control of the agent.
    Bensoussan et al. \cite{Bensoussan-Feng-Huang-21} considered an LQ-MFG with partial observation and common noise.

    The Stackelberg differential game problem, also known as the leader-follower differential game problem, arises in markets where certain companies possess greater authority to dominate others or individual entities.
    In response to this characteristic, Stackelberg \cite{Stackelberg-1952} introduced the concept of hierarchical solution. Within the context of the Stackelberg differential game, there are two players with asymmetric roles, one designated as the leader and the other as the follower.
    To attain a pair of Stackelberg equilibrium solutions, the differential game problem is typically bifurcated into two segments.
    In the first segment, the problem of the follower is addressed. Initially, the leader publicly announces their strategy, transforming the two-player differential game into the single-player optimal control problem for the follower.
    In other words, the follower promptly reacts, selecting an optimal strategy in response to the strategy disclosed by the leader, aiming to minimize (or maximize) their own cost functional.
    The second segment involves the leader choosing an optimal strategy, given the assumption that the follower will adopt such an optimal strategy, to minimize their own cost functional.
    This constitutes another optimal control problem for the leader. In summary, decision-making must be jointly accomplished by both players.
    Due to the asymmetry in roles, one player must be subordinate to the other, necessitating that one player makes their decision after the other has concluded their decision-making process.
    The Stackelberg differential game problem holds significant relevance in financial and economic practices, prompting an increasing focus in applied research.
    Bagchi and Ba\c{s}ar \cite{Bagchi-Basar-81} explored an LQ Stackelberg stochastic differential game, where the diffusion coefficient in the state equation does not involve state and control variables.
    Yong \cite{Yong-02} delved into a more generalized framework of LQ leader-follower differential game problems.
    In this study, coefficients of the state system and cost functional are stochastic, the diffusion coefficient in the state equation includes control variables, and the weight matrix in front of the control variables in the cost functional is not necessarily positive definite.
    Bensoussan et al. \cite{Bensoussan-Chen-Sethi-15} introduced several solution concepts based on players' information sets and investigated LQ Stackelberg differential games under adaptive open-loop and closed-loop memoryless information structures, where control variables do not enter the diffusion coefficient in the state equation.
    Zheng and Shi \cite{Zheng-Shi-20} investigated a Stackelberg game involving {\it backward stochastic differential equations} (BSDEs) (Pardoux and Peng \cite{Pardoux-Peng-90}, Ma and Yong \cite{Ma-Yong-99}).
    Feng et al. \cite{Feng-Hu-Huang-22} examined a Stackelberg game associated with BSDE featuring constraints.
    Sun et al. \cite{Sun-Wang-Wen-23} conducted research on a zero-sum LQ Stackelberg stochastic differential game.

    {\it Stackelberg mean field games} (Stackelberg MFGs), distinct from Stackelberg stochastic differential games of mean field type (incorporating the expected values of state and control variables, as seen in references \cite{Du-Wu-19}, \cite{Lin-Jiang-Zhang-19}, \cite{Wang-Zhang-20}, \cite{Moon-Yang-21}, \cite{Lin-Zhang-23}, etc.), have been increasingly capturing the attention of researchers.
    Nourian et al. \cite{Nourian-Caines-Malhami-Huang-12} studied a large population LQ leader-follower stochastic multi-agent systems and established their $(\epsilon_1, \epsilon_2)$-Stackelberg-Nash equilibrium.
    Bensoussan et al. \cite{Bensoussan-Chau-Yam-15} and Bensoussan et al. \cite{Bensoussan-Chau-Lai-Yam-17} investigated Stackelberg MFGs featuring delayed responses.
    Wang and Zhang \cite{Wang-Zhang-14} examined hierarchical games for multi-agent systems involving a leader and a large number of followers with infinite horizon tracking-type costs.
    Moon and Ba\c{s}ar \cite{Moon-Basar-18} considered the LQ Stackelberg MFG with the adapted open-loop information structure, and derived $(\epsilon_1, \epsilon_2)$-Stackelberg-Nash equilibrium.
    Yang and Huang \cite{Yang-Huang-21} conducted a study on LQ Stackelberg MFGs involving a major player (leader) and $N$ minor players (followers).
    Si and Wu \cite{Si-Wu-21} explored a backward-forward LQ Stackelberg MFG, where the leader's state equation is backward, and the followers' state equation is forward.
    Wang \cite{Wang-24} employed a direct method to solve LQ Stackelberg MFGs with a leader and a substantial number of followers.
    A static output feedback strategy for robust incentive Stackelberg games with a large population for mean field stochastic systems was investigated in Mukaidani et al. \cite{Mukaidani-Irie-Xu-Zhuang-22}.
    Dayanikli and Lauri\`{e}re \cite{Dayanikli-Lauriere-23} proposed a numerical approach of machine learning techniques to solve Stackelberg problems between a principal and a mean field of agents.

    Conventionally, two approaches are employed in the resolution of mean field games.
    One is termed the fixed-point approach (or top-down approach, NCE approach, see \cite{Huang-Caines-Malhame-07}, \cite{Huang-Malhame-Caines-06}, \cite{Li-Zhang-08}, \cite{Bensoussan-Frehse-Yam-13}, \cite{Carmona-Delarue-18}), which initiates the process by employing mean field approximation and formulating a fixed-point equation.
    By tackling the fixed-point equation and scrutinizing the optimal response of a representative player, decentralized strategies can be formulated.
    The alternative approach is known as the direct approach (or bottom-up approach, refer to \cite{Lasry-Lions-07}, \cite{Wang-Zhang-Zhang-20}, \cite{Huang-Zhou-20}, \cite{Wang-24}).
    This method commences by formally solving an $N$-player game problem within a vast and finite population setting.
    Subsequently, by decoupling or reducing high-dimensional systems, centralized control can be explicitly derived, contingent on the state of a specific player and the average state of the population.
    As the population size $N$ approaches infinity, the construction of decentralized strategies becomes feasible. In \cite{Huang-Zhou-20}, the authors addressed the connection and difference of these two routes in an LQ setting.

    In this paper, we explore a backward-forward LQ Stackelberg MFG with a leader and multiple followers.
    The leader initiates the process by disclosing their strategy, following which each follower optimizes its individual cost.
    Employing the direct approach, we formulate $(\epsilon_1, \epsilon_2)$-Stackelberg equilibrium strategies.
    With the leader's strategy given, we first address a MFG by the {\it stochastic maximum principle} (SMP), resulting in a system of {\it forward-backward stochastic differential equations} (FBSDEs).
    Subsequent to the followers implementing their strategies, the leader is encounters an optimal control problem driven by a forward SDE and two BSDEs.
    By variational analysis, we obtain the centralized strategy for the leader.
    By decoupling a high-dimensional FBSDE with mean field approximations, we construct a set of decentralized $(\epsilon_1, \epsilon_2)$-Stackelberg equilibrium strategies, in terms of a 6-dimensional FBSDE.

    The main contributions of the paper are outlined as follows.

    $\bullet$ We formulate a new type of LQ Stackelberg MFGs, where the state of leader is a backward equation, and followers are forward equations.
    Note that in our model, both the state and control variables of the leader enter into the drift coefficient of the forward equation, which is different from those in \cite{Huang-Wang-Wu-16}, \cite{Si-Wu-21}.
    Our model is also different with that in \cite{Wang-24}, where the leader's state is a forward equation.
    Moreover, the cross terms of the leader's and followers' control variables enter into the followers' cost functionals.
    These new structures bring mathematical difficulties, and, to some degree, they have some potential applications in reality.

    $\bullet$ In lieu of the traditional fixed-point methodology, we embrace a direct approach to tackle the complexities of our game problem, which was used in \cite{Huang-Zhou-20}, \cite{Wang-24}.
    This strategic shift involves the meticulous decoupling of the high-dimensional Hamiltonian system using mean field approximations.
    By disentangling the high-dimensional Hamiltonian system through mean field approximations, we formulate a set of decentralized strategies for all players. This set is subsequently demonstrated to constitute an $(\epsilon_1, \epsilon_2)$-Stackelberg equilibrium.

    The paper is structured as follows.
    In Section 2, we formulate the problem of backward-forward LQ Stackelberg MFG.
    In Section 3, we explore and determine centralized strategies for $N$ followers and the leader.
    In Section 4, we construct decentralized strategies and the proof of the asymptotic optimality is rigorously presented.
    Section 5 engages in a numerical simulation to demonstrate the effectiveness of our obtained results.
    Finally, some conclusions and future research directions are given in Section 6.

    The following notations will be used throughout this paper.
    We use $||\cdot||$ to denote the norm of a Euclidean space, or the Frobenius norm for matrices. $\top$ denotes the transpose of a vector or matrix. For a symmetric matrix $Q$ and a vector $z$, $||z||^2_Q \equiv z^\top Qz$.
    For any real-valued scalar functions $f(\cdot)$ and $g(\cdot)$ defined on $\mathbb{R}$, $f(x) = O(g(x))$ means that there exists a constant $C > 0$ such that $\lim_{x \to \infty} |\frac{f(x)}{g(x)}| = C$, where $|\cdot|$ is an absolute value, which is also equivalent to saying that there exist $C > 0$ and $x$ such that $|f(x)| \le C|g(x)|$ for any $x \ge x'$.

    Let $T>0$ be a finite time duration and $(\Omega, \mathcal{F}, \{\mathcal{H}_t\}_{0 \le t \le T}, \mathbb{P})$ be a complete filtered probability space with the filtration $\{\mathcal{H}_t\}_{0 \le t \le T}$ augmented by all the $\mathbb{P}$-null sets in $\mathcal{F}$. $\mathbb{E}$ denoted the expectation with respect to $\mathbb{P}$. Let $L_{\mathcal{H}}^2(0,T;\cdot)$ be the set of all vector-valued (or matrix-valued) $\mathcal{H}_t$-adapted processes $f(\cdot)$ such that $\mathbb{E}\big[\int^T_0 ||f(t)||^2dt\big] < \infty$ and $L^2_{\mathcal{H}_t}(\Omega;\cdot)$ be the set of $\mathcal{H}_t$-measurable random variables, for $t\in[0,T]$.

    \section{Problem formulation}

    We consider a large-population system with one leader and $N$ followers, where $N$ can be arbitrarily large. The states equation of the leader and the $i$th follower, $1 \le i \le N$, are given by the following controlled linear BSDE and SDE, respectively:
    \begin{equation}\label{state}
        \left\{
        \begin{aligned}
            dx_0(t) &= \big[A_0x_0(t) + B_0u_0(t) + C_0z_0(t) + f_0\big]dt + z_0(t)dW_0(t),\\
            dx_i(t) &= \big[Ax_i(t) + Bu_i(t) + Fx_0(t) + Gu_0(t) + f\big]dt + DdW_i(t),\quad t\in[0,T],\\
             x_0(T) &= \xi_0 ,\quad x_i(0) = \xi_i , \quad i = 1,\cdots,N,
        \end{aligned}
        \right.
    \end{equation}
    where $x_0(\cdot)$, $u_0(\cdot)$ are the state process and the control process, $\xi_0 \in L^2_{\mathcal{F}_T^0}(\Omega; \mathbb{R})$ is the given terminal value of the leader; similarly, $x_i(\cdot)$, $u_i(\cdot)$ and $\xi_i$ are the state process, control process and initial value (random variable) of the $i$th follower.
    In this paper, for simplicity, we assume the dimensions of state process and control process are both one-dimensional.
    Here, $A_0, B_0, C_0, f_0, A, B, F, G, f, D$ are scalar constants.
    $W_i(\cdot), i = 0,\cdots,N$ are a sequence of one-dimensional Brownian motions defined on $(\Omega, \mathcal{F}, \{\mathcal{F}_t\}_{0 \le t \le T}, \mathbb{P})$.
    Let $\mathcal{F}_t$ be the $\sigma$-algebra generated by $\{\xi_i, W_0(s), W_i(s), s \le t, 1 \le i \le N\}$.
    Denote $\mathcal{F}_t^0$ be the $\sigma$-algebra generated by $\{W_0(s), s \le t\}$ and $\mathcal{F}_t^i$ be the $\sigma$-algebra generated by $\{\xi_i, W_0(s), W_i(s), s \le t\}, 1 \le i \le N$.
    Define the decentralized control set for the leader as
    $$\mathscr{U}_0[0,T] := \big\{u_0(\cdot)|u_0(\cdot) \in L^2_{\mathcal{F}^0}(0,T;\mathbb{R})\big\}.$$
    Define the decentralized control set for the followers and $i$th follower as
    $$\mathscr{U}_d[0,T] := \big\{(u_1(\cdot),\cdots,u_N(\cdot))|u_i(\cdot) \in L^2_{\mathcal{F}^i}(0,T;\mathbb{R})\big\},$$
    $$\mathscr{U}_i[0,T] := \big\{u_i(\cdot)|u_i(\cdot) \in L^2_{\mathcal{F}^i}(0,T;\mathbb{R})\big\},$$
    and the centralized control set for the followers as
    $$\mathscr{U}_c[0,T] := \big\{(u_1(\cdot),\cdots,u_N(\cdot))|u_i(\cdot) \in L^2_{\mathcal{F}}(0,T;\mathbb{R}), 1 \le i \le N\big\}.$$

    The cost functional of the leader is given by
    \begin{equation}\label{cost of leader}
    \hspace{-0.3mm}    J^N_0(u_0(\cdot), u^N(\cdot)) = \frac{1}{2} \mathbb{E} \left\{\int_0^T \left[||x_0(t) - \Gamma_0x^{(N)}(t) - \eta_0||^2_{Q_0} + ||u_0(t)||^2_{R_0}\right]dt + ||x_0(0)||^2_{H_0}\right\},
    \end{equation}
    where $u^N(\cdot) := (u_1(\cdot),\cdots,u_N(\cdot))$, $x^{(N)}(\cdot):= \frac{1}{N} \sum_{i=1}^{N} x_i(\cdot)$ is called the state average or mean field term of all followers, $Q_0$, $R_0$, $H_0$, $\Gamma_0$ and $\eta_0$ are scalar constants.
    Notice that $Q_0$, $\Gamma_0$ and $\eta_0$ determine the coupling between the leader and the mean field of the $N$ followers, $R_0$ is the control performance weighting parameter of the leader, and $H_0$ is the initial state cost parameter of the leader.

    For the $i$th follower, the cost functional is defined by
    \begin{equation}\label{cost of followers}
    \begin{aligned}
        &J^N_i(u_i(\cdot), u_{-i}(\cdot), u_0(\cdot)) \\
        &= \frac{1}{2} \mathbb{E} \left\{\int_0^T \left[||x_i(t) - \Gamma x^{(N)}(t) - \Gamma_1 x_0(t) - \eta||^2_Q + ||u_i(t)||^2_R + 2u_i(t)Lu_0(t)\right]dt + ||x_i(T)||^2_H\right\},
    \end{aligned}
    \end{equation}
    where $u_{-i}(\cdot) := (u_1(\cdot),\cdots,u_{i-1}(\cdot),u_{i+1}(\cdot),\cdots,u_N(\cdot))$, $Q$, $R$, $L$, $H$, $\Gamma$, $\Gamma_1$ and $\eta$ are scalar constants.
    $Q$, $\Gamma$, $\Gamma_1$ and $\eta$ determine the coupling between the $i$th follower, leader and followers' mean field, whereas $L$ influences the coupling between the leader and the $i$th follower through control term.
    Also, $R$ serves as the control performance weighting parameter of the $i$th follower, and $H$ represents the terminal state cost parameter.
    It is noteworthy that the followers are (weakly) coupled with each other through the mean field term $x^{(N)}$, and are (strongly) coupled with the leader's state $x_0$ and control $u_0$ included in their cost functionals.

    We posit the following assumptions:

    \noindent {\bf (A1)} $\{\xi_i\}, i = 1,2,\cdots,N$ are a sequence of independent and identically distributed (i.i.d., for short) random variables with $\mathbb{E}[\xi_i] = \bar{\xi}, i = 1,2,\cdots,N$, and there exists a constant $c$ such that $\sup_{1 \le i \le N} \mathbb{E} [||\xi_i||^2] \le c.$

    \noindent {\bf (A2)} $\{W_i(t), 0 \le i \le N \}$ are independent of each other, which are
    also independent of $\{\xi_i, 1 \le i \le N \}$.

    \noindent {\bf (A3)} $Q_0 \ge 0, R_0 > 0, H_0 \ge 0$ and $Q \ge 0, R > 0, H \ge 0$.

    Now, we present the precise definition of an $(\epsilon_1, \epsilon_2)$-Stackelberg equilibrium.

    \begin{mydef}\label{def2.1}
        A set of strategies $(u_0^*(\cdot),u_1^*(\cdot),\cdots,u_N^*(\cdot))$ is an $(\epsilon_1, \epsilon_2)$-Stackelberg equilibrium with respect to $\{J^N_i,0 \le i \le N\}$ if the following hold:

        (i) For a given strategy of the leader $u_0(\cdot) \in \mathscr{U}_0[0,T]$, $u^{N*}(\cdot) = (u_1^*(\cdot),\cdots,u_N^*(\cdot))$ constitutes an $\epsilon_1$-Nash equilibrium, if there exists a constant $\epsilon_1 \ge 0$ such that for all $i, 1 \le i \le N$,
        $$J^N_i(u_i^*(\cdot), u_{-i}^*(\cdot), u_0(\cdot)) \le \inf_{u_i(\cdot) \in \mathscr{U}_i[0,T]} J^N_i(u_i(\cdot), u_{-i}^*(\cdot), u_0(\cdot)) + \epsilon_1;$$

        (ii) There exists a constant $\epsilon_2 \ge 0$ such that
        $$J^N_0(u_0^*(\cdot), u^{N*}[\cdot;u_0^*(\cdot)]) \le \inf_{u_0(\cdot) \in \mathscr{U}_0[0,T]} J^N_0(u_0(\cdot), u^{N*}[\cdot;u_0(\cdot)]) + \epsilon_2.$$
    \end{mydef}

    In this paper, we investigate the following problem.

    \noindent {\bf (P0)}: Find an $(\epsilon_1, \epsilon_2)$-Stackelberg equilibrium solution to (\ref{cost of leader}), (\ref{cost of followers}), subject to (\ref{state}).

    \section{Centralized stratigies}

    \subsection{Mean field Nash games for the $N$ followers}

    In this subsection, we consider the mean field Nash game for the $N$ followers under an arbitrary strategy of the leader $u_0(\cdot) \in \mathscr{U}_0[0,T]$.
    We suppose $u_0(\cdot)$ is fixed. Then, due to first equation of ({\ref{state}}), $x_0(\cdot)$ is also fixed.
    We now consider the following game problem for $N$ followers.

    \noindent {\bf (P1)}: Minimize $J^N_i,i = 1,\cdots,N$ of (\ref{cost of followers}) over $u(\cdot) \in \mathscr{U}_c[0,T]$.

    For the sake of simplicity, the time variable $t$ will be omitted without ambiguity.

    We have the following result.

    \begin{mypro}\label{prop3.1}
        Under Assumption (A1)-(A3), let $u_0(\cdot) \in \mathscr{U}_0[0,T]$ be given, for the initial value $\xi_i, i = 1,\cdots,N$, {\bf (P1)} admits an optimal control $\check{u}_i(\cdot) \in L^2_{\mathcal{F}}(0,T;\mathbb{R}), i = 1,\cdots,N$, if and only if the following two conditions hold:

        (i) The adapted solution $(\check{x}_i(\cdot), \check{p}_i(\cdot), \check{q}^j_i(\cdot), i = 1,\cdots,N, j = 0,1,\cdots,N)$ to the FBSDE
        \begin{equation}\label{adjoint FBSDE of followers}
            \left\{
            \begin{aligned}
                d\check{x}_i &= \big[A\check{x}_i + B\check{u}_i + Fx_0 + Gu_0 + f\big]dt + DdW_i,\\
                d\check{p}_i &= -\left[A\check{p}_i + \left(1 - \frac{\Gamma}{N}\right)Q\left(\check{x}_i - \Gamma \check{x}^{(N)} - \Gamma_1 x_0 - \eta\right)\right]dt + \sum_{j=0}^{N} \check{q}_i^j dW_j,\quad t\in[0,T],\\
                \check{x}_i(0) &= \xi_i ,\quad \check{p}_i(T) = H\check{x}_i(T) , \quad i = 1,\cdots,N,
            \end{aligned}
            \right.
        \end{equation}
        satisfies the following stationarity condition:
        \begin{equation}\label{stationarity condition of followers}
            B\check{p}_i + R\check{u}_i + Lu_0 = 0, \quad a.e.,\ a.s., \quad i = 1,\cdots,N.
        \end{equation}

        (ii) For $i = 1,\cdots,N$, the following convexity condition holds:
        \begin{equation}\label{convexity condition of followers}
            \mathbb{E} \left\{\int_0^T \left[Q \left(1 - \frac{\Gamma}{N}\right)^2 \tilde{x}_i^2 + Ru_i^2\right]dt + H \tilde{x}_i^2(T)\right\} \ge 0,\ i = 1,\cdots,N, \quad \forall u_i(\cdot) \in L^2_{\mathcal{F}}(0,T;\mathbb{R}),
        \end{equation}
        where $\tilde{x}_i(\cdot)$ is the solution to the following {\it random differential equation} (RDE):
        \begin{equation}\label{random differential equation of followers}
            \left\{
            \begin{aligned}
                &d\tilde{x}_i = [A\tilde{x}_i + Bu_i]dt,\\
                &\tilde{x}_i(0) = 0.
            \end{aligned}
            \right.
        \end{equation}
        \label{pro1}
    \end{mypro}

    \begin{proof}
        We consider the $i$th follower.
        For given $\xi_i \in L^2_{\mathcal{F}_0^i}(\Omega; \mathbb{R})$, $u_0(\cdot) \in \mathscr{U}_0[0,T]$, $\check{u}_i(\cdot) \in L^2_{\mathcal{F}}(0,T;\mathbb{R})$, suppose that $(\check{x}_i(\cdot), \check{p}_i(\cdot), \check{q}^j_i(\cdot), j = 0,1,\cdots,N)$ is an adapted solution to FBSDE (\ref{adjoint FBSDE of followers}).
        For any $u_i(\cdot) \in L^2_{\mathcal{F}}(0,T;\mathbb{R})$ and $\varepsilon \in \mathbb{R}$, let $x_i^{\varepsilon}(\cdot)$ be the solution to the following perturbed state equation:
        \begin{equation}\label{perturbed state of followers}
            \left\{
            \begin{aligned}
                dx_i^{\varepsilon} &= \big[Ax_i^{\varepsilon} + B(\check{u}_i + \varepsilon u_i) + Fx_0 + Gu_0 + f\big]dt + DdW_i,\\
                x_i^{\varepsilon}(0) &= \xi_i.
            \end{aligned}
            \right.
        \end{equation}
        Then, $\tilde{x}_i^{\varepsilon}(\cdot) := \frac{x_i^{\varepsilon}(\cdot) - \check{x}_i(\cdot)}{\varepsilon}$ is independent of $\varepsilon$ and satisfies (\ref{random differential equation of followers}).

        Applying It\^{o}'s formula to $\check{p}_i(\cdot) \tilde{x}_i(\cdot)$, integrating from 0 to $T$, and taking the expectation, we have
        \begin{equation*}
        \begin{aligned}
            &\mathbb{E} \big[H\check{x}_i(T) \tilde{x}_i(T)\big] = \mathbb{E} \big[\check{p}_i(T) \tilde{x}_i(T) - \check{p}_i(0) \tilde{x}_i(0)\big] \\
            &= \mathbb{E} \int_0^T \left\{-\left[A\check{p}_i + \left(1 - \frac{\Gamma}{N}\right)Q\left(\check{x}_i - \Gamma \check{x}^{(N)} - \Gamma_1 x_0 - \eta\right)\right]\tilde{x}_i + \big[A\tilde{x}_i + Bu_i\big]\check{p}_i\right\}dt \\
            &= \mathbb{E} \int_0^T \left\{-\left(1 - \frac{\Gamma}{N}\right)Q\left(\check{x}_i - \Gamma \check{x}^{(N)} - \Gamma_1 x_0 - \eta\right)\tilde{x}_i + Bu_i\check{p}_i\right\}dt.
        \end{aligned}
        \end{equation*}
        Hence,
        \begin{equation*}
        \begin{aligned}
            &J^N_i(\check{u}_i(\cdot) + \varepsilon u_i(\cdot), u_{-i}(\cdot), u_0(\cdot)) - J^N_i(\check{u}_i(\cdot), u_{-i}(\cdot), u_0(\cdot))\\
            &= \frac{1}{2} \mathbb{E} \left\{\int_0^T \left[||x_i^{\varepsilon} - \Gamma x^{\varepsilon(N)} - \Gamma_1 x_0 - \eta||^2_Q + ||\check{u}_i + \varepsilon u_i||^2_R + 2(\check{u}_i + \varepsilon u_i)Lu_0\right]dt + ||x_i^{\varepsilon}(T)||^2_H\right\}\\
            &\quad - \frac{1}{2} \mathbb{E} \left\{\int_0^T \left[||\check{x}_i - \Gamma \check{x}^{(N)} - \Gamma_1 x_0 - \eta||^2_Q + ||\check{u}_i||^2_R + 2\check{u}_iLu_0\right]dt + ||\check{x}_i(T)||^2_H\right\} \\
            &= \frac{1}{2} \mathbb{E} \left\{\int_0^T \left[||\left(1 - \frac{\Gamma}{N}\right)(\check{x}_i + \varepsilon \tilde{x}_i) - \Gamma x^{(N-1)}_{-i} - \Gamma_1 x_0 - \eta||^2_Q + ||\check{u}_i + \varepsilon u_i||^2_R \right. \right.\\
            &\qquad\qquad + 2(\check{u}_i + \varepsilon u_i)Lu_0\bigg]dt+ ||(\check{x}_i(T) + \varepsilon \tilde{x}_i(T))||^2_H\bigg\} \\
            &\quad - \frac{1}{2} \mathbb{E} \left\{\int_0^T \left[||\left(1 - \frac{\Gamma}{N}\right)\check{x}_i - \Gamma x^{(N-1)}_{-i} - \Gamma_1 x_0 - \eta||^2_Q + ||\check{u}_i||^2_R + 2\check{u}_iLu_0\right]dt + ||\check{x}_i(T)||^2_H\right\} \\
            &= \frac{1}{2} \varepsilon^2 \mathbb{E} \left\{\int_0^T \bigg[Q\left(1 - \frac{\Gamma}{N}\right)^2 \tilde{x}_i^2 + Ru_i^2\bigg]dt + H \tilde{x}_i^2(T)\right\} \notag \\
            &\quad + \varepsilon \mathbb{E} \left\{\int_0^T \left[Q\left(1 - \frac{\Gamma}{N}\right) \tilde{x}_i \left(\check{x}_i - \Gamma \check{x}^{(N)} - \Gamma_1 x_0 - \eta\right) + u_iR\check{u}_i + u_iLu_0\right]dt + H \check{x}_i(T) \tilde{x}_i(T)\right\} \\
            &= \frac{1}{2} \varepsilon^2 \mathbb{E} \left\{\int_0^T \bigg[Q\left(1 - \frac{\Gamma}{N}\right)^2 \tilde{x}_i^2 + Ru_i^2\bigg]dt + H \tilde{x}_i^2(T)\right\} + \varepsilon \mathbb{E} \left\{\int_0^T u_i\left[B\check{p}_i + R\check{u}_i + Lu_0\right]dt\right\},
        \end{aligned}
        \end{equation*}
        where only here $x^{\varepsilon(N)}(\cdot) := \frac{x_i^{\varepsilon}(\cdot)}{N} + \frac{1}{N} \sum_{j \ne i} x_j(\cdot)$, $\check{x}^{(N)}(\cdot) := \frac{\check{x}_i(\cdot)}{N} + \frac{1}{N} \sum_{j \ne i} x_j(\cdot)$, and $x^{(N-1)}_{-i}(\cdot) = \frac{1}{N} \sum_{j \ne i} x_j(\cdot)$.
        Therefore,
        $$J^N_i(\check{u}_i(\cdot), u_{-i}(\cdot), u_0(\cdot)) \le J^N_i(\check{u}_i(\cdot) + \varepsilon u_i(\cdot), u_{-i}(\cdot), u_0(\cdot))$$
        if and only if (\ref{stationarity condition of followers}) and (\ref{convexity condition of followers}) hold. The proof is complete.
    \end{proof}

    Based on Assumption (A3), we can figure out that the open-loop optimal strategies for followers are
    \begin{equation}\label{open-loop optimal strategies of followers}
        \check{u}_i = -R^{-1}[B\check{p}_i + Lu_0],\quad a.e.,\ a.s.,\quad  i = 1,\cdots,N,
    \end{equation}
    so the related Hamiltonian system can be represented by
    \begin{equation}\label{Hamiltonian system of followers}
        \left\{
        \begin{aligned}
            d\check{x}_i &= \big[A\check{x}_i - BR^{-1}B\check{p}_i + Fx_0 + (G - BR^{-1}L)u_0 + f\big]dt + DdW_i,\\
            d\check{p}_i &= -\left[A\check{p}_i + \left(1 - \frac{\Gamma}{N}\right)Q\left(\check{x}_i - \Gamma \check{x}^{(N)} - \Gamma_1 x_0 - \eta\right)\right]dt + \sum_{j=0}^{N} \check{q}_i^j dW_j,\\
            \check{x}_i(0) &= \xi_i ,\quad \check{p}_i(T) = H\check{x}_i(T) , \quad i = 1,\cdots,N.
        \end{aligned}
        \right.
    \end{equation}

    We consider the transformation $\check{p}_i(\cdot) = P_N(\cdot)\check{x}_i(\cdot) + K_N(\cdot)\check{x}^{(N)}(\cdot) + \check{\phi}_N(\cdot), i = 1,\cdots,N$, where $P_N(\cdot)$, $K_N(\cdot)$ are differential functions with $P_N(T) = H$, $K_N(T) = 0$, process pair $(\phi_N(\cdot),V_N(\cdot))$ satisfies a BSDE as follows:
    $$d\phi_N(t)= \Lambda_N(t)dt + V_N(t)dW_0(t), \quad \phi_N(T) = 0,$$
    where process $\Lambda_N(\cdot)$ will be determined later.
    By It\^{o}'s formula, we get
    \begin{equation*}
    \begin{aligned}
        d\check{p}_i &= \dot{P}_N \check{x}_i dt + P_N\Big\{\big[A\check{x}_i - BR^{-1}B\big(P_N\check{x}_i + K_N\check{x}^{(N)} + \check{\phi}_N\big) + Fx_0 + (G - BR^{-1}L)u_0 + f\big]dt  \\
        &\quad  + DdW_i\Big\} + \dot{K}_N \check{x}^{(N)} dt + K_N\Big\{\big[A\check{x}^{(N)} - BR^{-1}B\big((P_N + K_N)\check{x}^{(N)} + \check{\phi}_N\big) + Fx_0 \\
        &\quad + (G - BR^{-1}L)u_0 + f\big]dt + \frac{1}{N} \sum_{j=1}^{N}DdW_j\Big\} + \check{\Lambda}_Ndt + \check{V}_NdW_0 \notag \\
        &= -\left[A\left(P_N\check{x}_i + K_N\check{x}^{(N)} + \check{\phi}_N\right) + \left(1 - \frac{\Gamma}{N}\right)Q\left(\check{x}_i - \Gamma \check{x}^{(N)} - \Gamma_1 x_0 - \eta\right)\right]dt + \sum_{j=0}^{N} \check{q}_i^j dW_j,
    \end{aligned}
    \end{equation*}
    where the last equality is due to the second equation of (\ref{Hamiltonian system of followers}).
    Comparing the coefficients of the corresponding terms, we have
    $$\check{q}_i^i = P_ND + \frac{K_ND}{N},\quad \check{q}_i^j = \frac{K_ND}{N},\quad 1 \le j \ne i \le N, \quad \check{q}_i^0 = \check{V}_N,$$
    \begin{equation}\label{Riccati-1 of followers}
        \dot{P}_N + AP_N + P_NA - P_NBR^{-1}BP_N + \left(1 - \frac{\Gamma}{N}\right)Q = 0,\quad P_N(T) = H,
    \end{equation}
    \begin{equation}\label{Riccati-2 of followers}
        \dot{K}_N + AK_N + K_NA - P_NBR^{-1}BK_N - K_NBR^{-1}B(P_N + K_N) - \left(1 - \frac{\Gamma}{N}\right)Q\Gamma = 0,\quad K_N(T) = 0,
    \end{equation}
    \begin{equation}\label{BSDE of followers}
    \begin{aligned}
        d\check{\phi}_N &= -\bigg\{\left[A - (P_N + K_N)BR^{-1}B\right]\check{\phi}_N + (P_N + K_N)(G - BR^{-1}L)u_0 + (P_N + K_N)Fx_0 \\
        &\quad + (P_N + K_N)f - \left(1 - \frac{\Gamma}{N}\right)Q(\Gamma_1x_0 + \eta)\bigg\}dt + \check{V}_NdW_0,\quad \phi_N(T) = 0.
    \end{aligned}
    \end{equation}
    Let $\Pi_N(\cdot) := P_N(\cdot) + K_N(\cdot)$, then $\Pi_N(\cdot)$ satisfies
    \begin{equation}\label{Riccati-1+2 of followers}
        \dot{\Pi}_N + A\Pi_N + \Pi_NA - \Pi_NBR^{-1}B\Pi_N + \left(1 - \frac{\Gamma}{N}\right)Q - \left(1 - \frac{\Gamma}{N}\right)Q\Gamma = 0,\quad \Pi_N(T) = H,
    \end{equation}
    thus (\ref{BSDE of followers}) can be rewritten as
    \begin{equation}\label{BSDE of followers-}
    \begin{aligned}
        d\check{\phi}_N &= -\bigg\{\left[A - \Pi_NBR^{-1}B\right]\check{\phi}_N + \Pi_N(G - BR^{-1}L)u_0 + \Pi_NFx_0  \\
        &\quad + \Pi_Nf - \left(1 - \frac{\Gamma}{N}\right)Q(\Gamma_1x_0 + \eta)\bigg\}dt + \check{V}_NdW_0,\quad \phi_N(T) = 0.
    \end{aligned}
    \end{equation}

    Note that (\ref{Riccati-1 of followers}) is a symmetric Riccati differential equation, if it satisfies Assumption (A3) and $1 - \frac{\Gamma}{N} \ge 0$, then it admits a unique solution.
    Similarly, (\ref{Riccati-1+2 of followers}) is a symmetric Riccati differential equation, if it satisfies Assumption (A3) and $(1 - \frac{\Gamma}{N})(1 - \Gamma) \ge 0$, then it admits a unique solution.
    Thus (\ref{Riccati-2 of followers}) admits a unique solution.
    And then (\ref{BSDE of followers-}) admits a unique solution.

    From the above discussion, Proposition \ref{pro1} and Theorem 4.1 on page 47 of Ma and Yong \cite{Ma-Yong-99}, we have the following result.
    \begin{mythm}\label{thm3.1}
        Under Assumptions (A1)-(A3), for given $u_0(\cdot) \in \mathscr{U}_0[0,T]$, if $1 - \frac{\Gamma}{N} \ge 0$ and $(1 - \frac{\Gamma}{N})(1 - \Gamma) \ge 0$, then Problem {\bf (P1)} admits a unique solution
    \begin{equation}\label{optimal strategy of followers}
     \check{u}_i = -R^{-1}\left[B\big(P_N\check{x}_i + K_N\check{x}^{(N)} + \check{\phi}_N\big) + Lu_0\right],\quad i = 1,\cdots,N.
    \end{equation}
    \end{mythm}

    \subsection{Optimal strategy of the leader}

    Upon implementing the strategies of followers $\check{u}_i(\cdot), i = 1,\cdots,N$ according to (\ref{optimal strategy of followers}), we delve into an optimal control problem for the leader.

    {\bf (P2)}: Minimize $\check{J}^N_0(u_0(\cdot))$ over $u_0(\cdot) \in \mathscr{U}_0[0,T]$, where
    \begin{equation}\label{cost of leader-}
        \check{J}^N_0(u_0(\cdot)) := \frac{1}{2} \mathbb{E} \left\{\int_0^T \left[||x_0(t) - \Gamma_0\check{x}^{(N)}(t) - \eta_0||^2_{Q_0} + ||u_0(t)||^2_{R_0}\right]dt + ||x_0(0)||^2_{H_0}\right\},
    \end{equation}
    subject to
    \begin{equation}\label{state of leader-}
    \left\{
    \begin{aligned}
        dx_0 &= \big[A_0x_0 + B_0u_0 + C_0z_0 + f_0\big]dt + z_0dW_0,\quad x_0(T) = \xi_0,\\
          d\check{x}_i &= \big[(A - BR^{-1}BP_N)\check{x}_i - BR^{-1}BK_N\check{x}^{(N)} - BR^{-1}B\check{\phi}_N + Fx_0  \\
        &\quad + (G - BR^{-1}L)u_0 + f\big]dt + DdW_i,\quad \check{x}_i(0) = \xi_i, \quad i = 1,\cdots,N,\\
         d\check{\phi}_N &= -\bigg[(A - \Pi_NBR^{-1}B)\check{\phi}_N + \Pi_N(G - BR^{-1}L)u_0 + \Pi_NFx_0  \\
        &\quad + \Pi_Nf - \left(1 - \frac{\Gamma}{N}\right)Q(\Gamma_1x_0 + \eta)\bigg]dt + \check{V}_NdW_0,\quad \check{\phi}_N(T) = 0.
    \end{aligned}
    \right.
    \end{equation}
    Note that here $\check{x}^{(N)}(\cdot)$ means that all $x_i(\cdot), i = 1,\cdots,N$ take the optimal $\check{x}_i(\cdot)$, i.e., $\check{x}^{(N)}(\cdot) = \frac{1}{N} \sum_{i=1}^{N} \check{x}_i(\cdot)$.
    Denote $W^{(N)}(\cdot) := \frac{1}{N} \sum_{i=1}^{N} W_i(\cdot)$, and $\xi^{(N)} := \frac{1}{N} \sum_{i=1}^{N} \xi_i$.

    The result of this subsection is as follows.

    \begin{mythm}\label{thm3.2}
    Under Assumptions (A1)-(A3), let the followers adopt the optimal strategy (\ref{optimal strategy of followers}). Then Problem {\bf (P2)} admits an optimal control $\check{u}_0(\cdot) \in \mathscr{U}_0[0,T]$, if and only if the following two conditions hold:

    (i) The adapted solution $(\check{x}_0(\cdot), \check{z}_0(\cdot), \check{x}^{(N)}(\cdot), \check{\phi}_N(\cdot), \check{V}_N(\cdot), \check{y}_0(\cdot), \check{y}^{(N)}(\cdot), \check{\beta}^{(N)}(\cdot), \check{\psi}_N(\cdot))$ to the FBSDE
    \begin{equation}\label{adjoint FBSDE of leader}
    \left\{
    \begin{aligned}
                d\check{x}_0 = &\ \big[A_0\check{x}_0 + B_0\check{u}_0 + C_0\check{z}_0 + f_0\big]dt + \check{z}_0dW_0, \quad x_0(T) = \xi_0, \\
                d\check{x}^{(N)} = &\ \big[(A - BR^{-1}B\Pi_N)\check{x}^{(N)} - BR^{-1}B\check{\phi}_N + F\check{x}_0 + (G - BR^{-1}L)\check{u}_0 + f\big]dt \\
                &+ DdW^{(N)}, \quad \check{x}^{(N)}(0) = \xi^{(N)}, \\
                d\check{\phi}_N = &\ -\bigg[(A - \Pi_NBR^{-1}B)\check{\phi}_N + \Pi_N(G - BR^{-1}L)\check{u}_0 + \Pi_NF\check{x}_0 \\
                &+ \Pi_Nf - \left(1 - \frac{\Gamma}{N}\right)Q(\Gamma_1\check{x}_0 + \eta)\bigg]dt + \check{V}_NdW_0, \quad \check{\phi}_N(T) = 0, \\
                d\check{y}_0 = &\ -\left[Q_0(\check{x}_0 - \Gamma_0 \check{x}^{(N)} - \eta_0) + A_0\check{y}_0 + F\check{y}^{(N)} - \left(\Pi_NF - \left(1 - \frac{\Gamma}{N}\right)Q\Gamma_1\right)\check{\psi}_N\right]dt \\
                &\ - C_0\check{y}_0 dW_0, \quad \check{y}_0(0) = H_0\check{x}_0(0), \\
                d\check{y}^{(N)} = &\ -\big[(A - \Pi_NBR^{-1}B)\check{y}^{(N)} - Q_0\Gamma_0(\check{x}_0 - \Gamma_0\check{x}^{(N)} - \eta_0)\big]dt \\
                &\ + \check{\beta}^{(N)} dW^{(N)}, \quad \check{y}^{(N)}(T) = 0, \\
                d\check{\psi}_N = &\ \big[BR^{-1}B\check{y}^{(N)} + (A - \Pi_NBR^{-1}B)\check{\psi}_N\big]dt, \quad \check{\psi}_N(0) = 0,
    \end{aligned}
    \right.
    \end{equation}
    satisfies the following stationarity condition:
    \begin{equation}\label{stationarity condition of leader}
    B_0\check{y}_0 + R_0\check{u}_0 + (G - BR^{-1}L)\check{y}^{(N)} - \Pi_N(G - BR^{-1}L)\check{\psi}_N = 0, \quad a.e.,\ a.s..
    \end{equation}

    (ii) The following convexity condition holds:
    \begin{equation}\label{convexity condition of leader}
            \mathbb{E} \left\{\int_0^T \left[Q_0(\tilde{x}_0 - \Gamma_0\tilde{x}^{(N)})^2 + R_0u_0^2\right]dt + H_0 \tilde{x}_0^2(0)\right\} \ge 0, \quad \forall u_0(\cdot) \in \mathscr{U}_0[0,T],
    \end{equation}
     where $(\tilde{x}_0(\cdot), \tilde{z}_0(\cdot), \tilde{x}^{(N)}(\cdot), \tilde{\phi}_N(\cdot), \tilde{V}_N(\cdot))$ is the solution to the following FBSDE
    \begin{equation}\label{random FBSDE of leader}
            \left\{
            \begin{aligned}
                d\tilde{x}_0 = &\ \big[A_0\tilde{x}_0 + B_0u_0 + C_0\tilde{z}_0\big]dt + \tilde{z}_0dW_0, \quad \tilde{x}_0(T) = 0, \\
                d\tilde{x}^{(N)} = &\ \big[(A - BR^{-1}B\Pi_N)\tilde{x}^{(N)} - BR^{-1}B\tilde{\phi}_N + F\tilde{x}_0 + (G - BR^{-1}L)u_0\big]dt, \quad \tilde{x}^{(N)}(0) = 0, \\
                d\tilde{\phi}_N = & -\left[(A - \Pi_NBR^{-1}B)\tilde{\phi}_N + \left(\Pi_NF - \left(1 - \frac{\Gamma}{N}\right)Q\Gamma_1\right)\tilde{x}_0\right.\\
                &\quad + \Pi_N(G - BR^{-1}L)u_0\bigg]dt + \tilde{V}_NdW_0, \quad \tilde{\phi}_N(T) = 0.
            \end{aligned}
            \right.
    \end{equation}
    \label{thm2}
    \end{mythm}

    \begin{proof}
    For given $\xi_0 \in L^2_{\mathcal{F}_T^0}(\Omega;\mathbb{R})$ and $\check{u}_0(\cdot) \in \mathscr{U}_0[0,T]$, let ($\check{x}_0(\cdot)$, $\check{z}_0(\cdot)$, $\check{x}^{(N)}(\cdot)$, $\check{\phi}_N(\cdot)$, $\check{V}_N(\cdot)$, $\check{y}_0(\cdot)$, $\check{y}^{(N)}(\cdot)$, $\check{\beta}^{(N)}(\cdot)$, $\check{\psi}_N(\cdot)$) be an adapted solution to FBSDE (\ref{adjoint FBSDE of leader}).
    For any $u_0(\cdot) \in \mathscr{U}_0[0,T]$ and $\varepsilon \in \mathbb{R}$, let ($x_0^{\varepsilon}(\cdot)$, $z_0^{\varepsilon}(\cdot)$, $x^{\varepsilon(N)}(\cdot)$, $\phi_N^{\varepsilon}(\cdot)$, $V_N^{\varepsilon}(\cdot)$) be the solution to the following perturbed state equation of the leader:
    \begin{equation}\label{perturbed state of leader}
    \left\{
    \begin{aligned}
                dx_0^{\varepsilon} = &\ \big[A_0x_0^{\varepsilon} + B_0(\check{u}_0 + \varepsilon u_0) + C_0x_0^{\varepsilon} + f_0\big]dt + z_0^{\varepsilon}dW_0, \quad x_0^{\varepsilon}(T) = \xi_0, \\
                dx^{\varepsilon(N)} = &\ \big[(A - BR^{-1}B\Pi_N)x^{\varepsilon(N)} - BR^{-1}B\phi^{\varepsilon} + Fx_0^{\varepsilon} + (G - BR^{-1}L)(\check{u}_0 + \varepsilon u_0)  \\
                &\quad + f\big]dt+ DdW^{(N)}, \quad x^{\varepsilon(N)}(0) = \xi^{(N)}, \\
                d\phi_N^{\varepsilon} = & -\bigg[(A - \Pi_NBR^{-1}B)\phi_N^{\varepsilon} + \Pi_N(G - BR^{-1}L)(\check{u}_0 + \varepsilon u_0) + \Pi_NFx_0^{\varepsilon} \\
                &+ \Pi_Nf - \left(1 - \frac{\Gamma}{N}\right)Q(\Gamma_1x_0^{\varepsilon} + \eta)\bigg]dt + V_N^{\varepsilon}dW_0, \quad \phi_N^{\varepsilon}(T) = 0. \\
    \end{aligned}
    \right.
    \end{equation}
    Then, denoting by $(\tilde{x}_0(\cdot), \tilde{z}_0(\cdot), \tilde{x}^{(N)}(\cdot), \tilde{\phi}_N(\cdot), \tilde{V}_N(\cdot))$ the solution to (\ref{random FBSDE of leader}), we have $x_0^{\varepsilon}(\cdot) = \check{x}_0(\cdot) + \varepsilon \tilde{x}_0(\cdot)$, $z_0^{\varepsilon}(\cdot) = \check{z}_0(\cdot) + \varepsilon \tilde{z}_0(\cdot)$, $x^{\varepsilon(N)}(\cdot) = \check{x}^{(N)}(\cdot) + \varepsilon \tilde{x}^{(N)}(\cdot)$, $\phi_N^{\varepsilon}(\cdot) = \check{\phi}_N(\cdot) + \varepsilon \tilde{\phi}_N(\cdot)$, $V_N^{\varepsilon}(\cdot) = \check{V}_N(\cdot) + \varepsilon \tilde{V}_N(\cdot)$, and
    \begin{equation*}
    \begin{aligned}
            &\check{J}^N_0(\check{u}_0(\cdot) + \varepsilon u_0(\cdot)) - \check{J}^N_0(\check{u}_0(\cdot))\\
            & =\frac{1}{2} \mathbb{E} \left\{\int_0^T \left[||x_0^{\varepsilon} - \Gamma_0\check{x}^{\varepsilon(N)} - \eta_0||^2_{Q_0} + ||\check{u}_0 + \varepsilon u_0||^2_{R_0}\right]dt + ||x_0^{\varepsilon}(0)||^2_{H_0}\right\} \\
            &\quad - \frac{1}{2} \mathbb{E} \left\{\int_0^T \left[||\check{x}_0 - \Gamma_0\check{x}^{(N)} - \eta_0||^2_{Q_0} + ||\check{u}_0||^2_{R_0}\right]dt + ||\check{x}_0(0)||^2_{H_0}\right\} \\
            & =\frac{1}{2} \mathbb{E} \left\{\int_0^T \left[||(\check{x}_0 + \varepsilon \tilde{x}_0) - \Gamma_0(\check{x}^{(N)} + \varepsilon \tilde{x}^{(N)}) - \eta_0||^2_{Q_0} + ||\check{u}_0 + \varepsilon u_0||^2_{R_0}\right]dt\right. \\
            &\qquad + ||\check{x}_0(0) + \varepsilon \tilde{x}_0(0)||^2_{H_0}\bigg\}- \frac{1}{2} \mathbb{E} \left\{\int_0^T \left[||\check{x}_0 - \Gamma_0\check{x}^{(N)} - \eta_0||^2_{Q_0} + ||\check{u}_0||^2_{R_0}\right]dt + ||\check{x}_0(0)||^2_{H_0}\right\} \\
            &= \frac{1}{2} \varepsilon^2 \mathbb{E} \left\{\int_0^T \left[Q_0(\tilde{x}_0 - \Gamma_0\tilde{x}^{(N)})^2 + R_0u_0^2\right]dt + H_0 \tilde{x}_0^2(0)\right\} \\
            &\quad + \varepsilon \mathbb{E} \left\{\int_0^T \left[Q_0(\tilde{x}_0 - \Gamma_0\tilde{x}^{(N)})(\check{x}_0 - \Gamma_0 \check{x}^{(N)} - \eta_0) + R_0\check{u}_0u_0\right]dt + H_0 \check{x}_0(0) \tilde{x}_0(0)\right\}.
    \end{aligned}\
    \end{equation*}
    On the other hand, applying It\^{o}'s formula to $\tilde{x}_0(\cdot)\check{y}_0(\cdot) + \tilde{x}^{(N)}(\cdot)\check{y}^{(N)}(\cdot) + \tilde{\phi}_N(\cdot)\check{\psi}_N(\cdot)$, integrating from 0 to $T$ and taking expectation, we obtain
    \begin{equation*}
    \begin{aligned}
            &\ - \mathbb{E} \big[H_0 \check{x}_0(0) \tilde{x}_0(0)\big] \notag \\
            &= \mathbb{E} \int_0^T \bigg[Q_0(\tilde{x}_0 - \Gamma_0\tilde{x}^{(N)})\left(\check{x}_0 - \Gamma_0 \check{x}^{(N)} - \eta_0\right) - B_0\check{y}_0u_0 - (G - BR^{-1}L)\check{y}^{(N)}u_0 \notag \\
            &\qquad + \Pi_N(G - BR^{-1}L)\check{\psi}_Nu_0\bigg]dt.
    \end{aligned}
    \end{equation*}
    Hence,
        $$\check{J}^N_0(\check{u}_0(\cdot)) \le \check{J}^N_0(\check{u}_0(\cdot) + \varepsilon u_0(\cdot))$$
    if and only if (\ref{stationarity condition of leader}) and (\ref{convexity condition of leader}) hold. The proof is complete.
    \end{proof}

    Based on Assumption (A3), we can compute out the optimal control of the leader is
    \begin{equation}\label{optimal strategy of leader}
        \check{u}_0 = -R_0^{-1}\big[B_0\check{y}_0 + (G - BR^{-1}L)\check{y}^{(N)} - \Pi_N(G - BR^{-1}L)\check{\psi}_N\big],\quad a.e.,\ a.s..
    \end{equation}
    So the related Hamiltonian system can be represented by
    \begin{equation}\label{Hamiltonian system of leader}
        \left\{
        \begin{aligned}
            d\check{x}_0 &= \Big\{A_0\check{x}_0 - B_0R_0^{-1}\big[B_0\check{y}_0 + (G - BR^{-1}L)\check{y}^{(N)} - \Pi_N(G - BR^{-1}L)\check{\psi}_N\big] \\
            &\quad + C_0\check{z}_0 + f_0\Big\}dt + \check{z}_0dW_0, \quad x_0(T) = \xi_0, \\
            d\check{x}^{(N)} &= \Big\{(A - BR^{-1}B\Pi_N)\check{x}^{(N)} - BR^{-1}B\check{\phi}_N + F\check{x}_0 - (G - BR^{-1}L)R_0^{-1}\big[B_0\check{y}_0 \\
            &\quad + (G - BR^{-1}L)\check{y}^{(N)} - \Pi_N(G - BR^{-1}L)\check{\psi}_N\big] + f\Big\}dt + DdW^{(N)}, \quad \check{x}^{(N)}(0) = \xi^{(N)}, \\
            d\check{\phi}_N &= -\Big\{(A - \Pi_NBR^{-1}B)\check{\phi}_N + \Pi_N(G - BR^{-1}L)R_0^{-1}\big[B_0\check{y}_0 + (G - BR^{-1}L)\check{y}^{(N)}  \\
            &\quad - \Pi_N(G- BR^{-1}L)\check{\psi}_N\big] + \Pi_NF\check{x}_0 + \Pi_Nf - (1 - \frac{\Gamma}{N})Q(\Gamma_1\check{x}_0 + \eta)\Big\}dt\\
            &\quad + \check{V}_NdW_0, \quad \check{\phi}_N(T) = 0, \\
            d\check{y}_0 &= -\bigg[Q_0\big(\check{x}_0 - \Gamma_0 \check{x}^{(N)} - \eta_0\big) + A_0\check{y}_0 + F\check{y}^{(N)} - \left(\Pi_NF - \left(1 - \frac{\Gamma}{N}\right)Q\Gamma_1\right)\check{\psi}_N\bigg]dt \\
            &\quad - C_0\check{y}_0 dW_0, \quad \check{y}_0(0) = H_0\check{x}_0(0), \\
            d\check{y}^{(N)} &= -\big[(A - \Pi_NBR^{-1}B)\check{y}^{(N)} - Q_0\Gamma_0\big(\check{x}_0 - \Gamma_0\check{x}^{(N)} - \eta_0\big)\big]dt \\
            &\quad + \check{\beta}^{(N)} dW^{(N)}, \quad \check{y}^{(N)}(T) = 0, \\
            d\check{\psi}_N &= [BR^{-1}B\check{y}^{(N)} + (A - \Pi_NBR^{-1}B)\check{\psi}_N]dt, \quad \check{\psi}_N(0) = 0.
        \end{aligned}
        \right.
    \end{equation}

    \section{Decentralized strategies and $(\epsilon_1, \epsilon_2)$-Stackelberg equilibria}

    Let $N \to \infty$ in (\ref{Riccati-1 of followers}) and (\ref{Riccati-2 of followers}), then $P_N(\cdot) \to \bar{P}(\cdot)$, $K_N(\cdot) \to \bar{K}(\cdot)$, where $\bar{P}(\cdot)$, $\bar{K}(\cdot)$ satisfy
    \begin{equation}\label{Riccati-1 of followers--}
        \dot{\bar{P}} + A\bar{P} + \bar{P}A - \bar{P}BR^{-1}B\bar{P} + Q = 0,\quad \bar{P}(T) = H,
    \end{equation}
    \begin{equation}\label{Riccati-2 of followers--}
        \dot{\bar{K}} + A\bar{K} + \bar{K}A - \bar{P}BR^{-1}B\bar{K} - \bar{K}BR^{-1}B(\bar{P} + \bar{K}) - Q\Gamma = 0,\quad \bar{K}(T) = 0.
    \end{equation}
    Let $\bar{\Pi}(\cdot) := \bar{P}(\cdot) + \bar{K}(\cdot)$, then $\bar{\Pi}(\cdot)$ satisfies
    \begin{equation}\label{Riccati-1+2 of followers--}
        \dot{\bar{\Pi}} + A\bar{\Pi} + \bar{\Pi}A - \bar{\Pi}BR^{-1}B\bar{\Pi} + Q - Q\Gamma = 0,\quad \bar{\Pi}(T) = H,
    \end{equation}
    From Assumption (A3), it follows that (\ref{Riccati-1 of followers--}) and (\ref{Riccati-2 of followers--}) admits a unique solution, respectively.

    Inspired by (\ref{Hamiltonian system of leader}), we consider
    \begin{equation}\label{Hamiltonian system of leader--}
        \left\{
        \begin{aligned}
            d\bar{x}_0 &= \Big\{A_0\bar{x}_0 - B_0R_0^{-1}\big[B_0\bar{y}_0 + (G - BR^{-1}L)\bar{y} - \bar{\Pi}(G - BR^{-1}L)\bar{\psi}\big] \\
            &\quad + C_0\bar{z}_0 + f_0\Big\}dt + \bar{z}_0dW_0, \quad x_0(T) = \xi_0, \\
            d\bar{x} &= \Big\{(A - BR^{-1}B\bar{\Pi})\bar{x} - BR^{-1}B\bar{\phi} + F\bar{x}_0 - (G - BR^{-1}L)R_0^{-1}\big[B_0\bar{y}_0 \\
            &\quad + (G - BR^{-1}L)\bar{y} - \bar{\Pi}(G - BR^{-1}L)\bar{\psi}\big] + f\Big\}dt, \quad \bar{x}(0) = \bar{\xi}, \\
            d\bar{\phi} &= -\Big\{(A - \bar{\Pi}BR^{-1}B)\bar{\phi} + \bar{\Pi}(G - BR^{-1}L)R_0^{-1}\big[B_0\bar{y}_0 + (G - BR^{-1}L)\bar{y}  \\
            &\quad - \bar{\Pi}(G- BR^{-1}L)\bar{\psi}\big] + \bar{\Pi}F\bar{x}_0 + \bar{\Pi}f - Q(\Gamma_1\bar{x}_0 + \eta)\Big\}dt + \bar{V}dW_0, \quad \bar{\phi}(T) = 0, \\
            d\bar{y}_0 &= -\big[Q_0(\bar{x}_0 - \Gamma_0 \bar{x} - \eta_0) + A_0\bar{y}_0 + F\bar{y} - (\bar{\Pi}F - Q\Gamma_1)\bar{\psi}\big]dt \\
            &\quad - C_0\bar{y}_0 dW_0, \quad \bar{y}_0(0) = H_0\bar{x}_0(0), \\
            d\bar{y} &= -\big[(A - \bar{\Pi}BR^{-1}B)\bar{y} - Q_0\Gamma_0(\bar{x}_0 - \Gamma_0\bar{x} - \eta_0)\big]dt, \quad \bar{y}(T) = 0, \\
            d\bar{\psi} &= \big[BR^{-1}B\bar{y} + (A - \bar{\Pi}BR^{-1}B)\bar{\psi}\big]dt, \quad \bar{\psi}(0) = 0.
        \end{aligned}
        \right.
    \end{equation}

    Denote $X := [\bar{x}_0, \bar{y}, \bar{\phi}]^\top$, $Y := [\bar{y}_0, \bar{x}, \bar{\psi}]^\top$, $Z := [\bar{z}_0, 0, \bar{V}]^\top$ and
    \begin{equation*}
     \begin{aligned}
        \mathcal{A}_1 &:=
        \begin{bmatrix}
            A_0 & - B_0R_0^{-1}(G - BR^{-1}L) & 0 \\
            Q_0\Gamma_0 & - A + BR^{-1}B\bar{\Pi} & 0 \\
            - \bar{\Pi}F + Q\Gamma_1 & \bar{\Pi}R_0^{-1}(G - BR^{-1}L)^2 & - A + BR^{-1}B\bar{\Pi} \\
        \end{bmatrix},\\
        \mathcal{B}_1 &:=
        \begin{bmatrix}
            B_0R_0^{-1}B_0 & 0 & - B_0R_0^{-1}\bar{\Pi}(G - BR^{-1}L) \\
            0 & Q_0\Gamma_0^2 & 0 \\
            - B_0R_0^{-1}\bar{\Pi}(G - BR^{-1}L) & 0 & \bar{\Pi}^2R_0^{-1}(G - BR^{-1}L)^2 \\
        \end{bmatrix},\\
        \mathcal{A}_2 &:=
        \begin{bmatrix}
            -Q_0 & -F & 0 \\
            F & -R_0^{-1}(G - BR^{-1}L)^2 & -BR^{-1}B \\
            0 & BR^{-1}B & 0 \\
        \end{bmatrix},\\
        \mathcal{B}_2 &:=
        \begin{bmatrix}
            A_0 & - Q_0\Gamma_0 & - \bar{\Pi}F + Q\Gamma_1 \\
            B_0R_0^{-1}(G - BR^{-1}L) & - A + BR^{-1}B\bar{\Pi} & -\bar{\Pi}R_0^{-1}(G - BR^{-1}L)^2 \\
            0 & 0 & - A + BR^{-1}B\bar{\Pi} \\
        \end{bmatrix},\\
        \mathcal{C} &:=
        \begin{bmatrix}
            C_0 & 0 & 0 \\
            0 & 0 & 0 \\
            0 & 0 & 0 \\
        \end{bmatrix},\quad
        \mathfrak{f}_0 :=
        \begin{bmatrix}
            f_0 \\
            - Q_0\Gamma_0\eta_0 \\
            Q\eta - \bar{\Pi}f \\
        \end{bmatrix},\quad
        \mathfrak{f} :=
        \begin{bmatrix}
            Q_0\eta_0 \\
            f \\
            0 \\
        \end{bmatrix}.
    \end{aligned}
    \end{equation*}
    With the above notions, we can rewrite (\ref{Hamiltonian system of leader--}) as
    \begin{equation}\label{Hamiltonian system of leader---}
        \left\{
        \begin{aligned}
            dX &= \big[\mathcal{A}_1X - \mathcal{B}_1Y + \mathcal{C}Z + \mathfrak{f}_0\big]dt + ZdW_0,\quad X(T) = [\xi_0, 0, 0]^\top, \\
            dY &= \big[\mathcal{A}_2X - \mathcal{B}_2Y + \mathfrak{f}\big]dt - \mathcal{C}YdW_0,\quad Y(0) = [H_0\bar{x}_0(0), \bar{\xi}, 0]^\top.
        \end{aligned}
        \right.
    \end{equation}

    Suppose $(X(\cdot), Y(\cdot), Z(\cdot))$ is an adapted solution to (\ref{Hamiltonian system of leader---}). We assume that $X(\cdot)$ and $Y(\cdot)$ are related by the following affine transformation
    \begin{equation}
        X(\cdot) = \Phi(\cdot) Y(\cdot) + \Psi(\cdot),
    \end{equation}
    where $\Phi(\cdot)$ and $\Psi(\cdot)$ are both differentiable functions, with $\Phi(T) = 0$ and $\Psi(T) = [\xi_0, 0, 0]^\top$.
    Next, by It\^{o}'s formula, we have
    \begin{equation}\label{Ito formula of leader}
        dX = \big[\dot{\Phi}Y + \dot{\Psi}\big]dt + \Phi\big[\mathcal{A}_2(\Phi Y + \Psi) - \mathcal{B}_2Y + \mathfrak{f}\big]dt - \Phi \mathcal{C}Y dW_0.
    \end{equation}
    Now, comparing (\ref{Ito formula of leader}) with the first equation in (\ref{Hamiltonian system of leader---}), it follows that
    \begin{equation}\label{compare 1}
        Z = - \Phi \mathcal{C}Y.
    \end{equation}
    We can rewrite the first equation in (\ref{Hamiltonian system of leader---}) as
    \begin{equation}\label{rewrite Hamiltonian system of leader---}
        dX = \big[\mathcal{A}_1X - \mathcal{B}_1Y - \mathcal{C}\Phi \mathcal{C}Y + \mathfrak{f}_0\big]dt - \Phi \mathcal{C}YdW_0,\quad X(T) = [\xi_0, 0, 0]^\top.
    \end{equation}
    This together with (\ref{Ito formula of leader}) gives
    \begin{equation}\label{Riccati of leader}
        \dot{\Phi} - \Phi \mathcal{B}_2 + \Phi \mathcal{A}_2 \Phi - \mathcal{A}_1 \Phi + \mathcal{B}_1 + \mathcal{C} \Phi \mathcal{C} = 0, \quad \Phi(T) = 0,
    \end{equation}
    \begin{equation}\label{compare 3}
        \dot{\Psi} + \Phi \mathcal{A}_2 \Psi + \Phi \mathfrak{f} - \mathcal{A}_1 \Psi - \mathfrak{f}_0 = 0, \quad \Psi(T) = [\xi_0, 0, 0]^\top.
    \end{equation}

    Note that the Riccati equation (\ref{Riccati of leader}) is nonsymmetric.
    By Theorem 4.1 on page 47 of \cite{Ma-Yong-99} again, if (\ref{Riccati of leader}) admits a solution $\Phi(\cdot)$, then FBSDE (\ref{Hamiltonian system of leader---}) admits a unique adapted solution $(X(\cdot), Y(\cdot), Z(\cdot))$.

    We are going to identify conditions for the existence of a unique solution to (\ref{Riccati of leader}).

    \begin{mypro}\label{prop4.1}
        Let the pair $(\alpha(\cdot), \beta(\cdot))$ be the solution to the following differential equation:
        $$\frac{d}{dt}
        \begin{pmatrix}
            \alpha(t) \\
            \beta(t) \\
        \end{pmatrix} =
        \begin{pmatrix}
            \mathcal{A}_1 & -\mathcal{B}_1 \\
            \mathcal{A}_2 & -\mathcal{B}_2 \\
        \end{pmatrix}
        \begin{pmatrix}
            \alpha(t) \\
            \beta(t) \\
        \end{pmatrix}, \quad
        \begin{pmatrix}
            \alpha(T) \\
            \beta(T) \\
        \end{pmatrix} =
        \begin{pmatrix}
            0_{3 \times 3} \\
            I_3 \\
        \end{pmatrix}.$$
        If $\beta(t)$ is invertible for all $t \in [0,T]$, then the Riccati equation (\ref{Riccati of leader}) has a unique solution $\Phi(t) = \alpha(t)\beta^{-1}(t)$ for all $t \in [0,T]$.
    \end{mypro}

    \begin{proof}
        Note that
        $$\frac{d\beta^{-1}(t)}{dt} = -\beta^{-1}(t) \frac{d\beta(t)}{dt} \beta^{-1}(t) = -\beta^{-1}(t) \mathcal{A}_2 \alpha(t) \beta^{-1}(t) + \beta^{-1}(t) \mathcal{B}_2,$$
        which implies
        \begin{align*}
            \frac{d\Phi(t)}{dt} &= \frac{d\alpha(t)}{dt} \beta^{-1}(t) + \alpha(t) \frac{d\beta^{-1}(t)}{dt} \\
            &= \mathcal{A}_1 \alpha(t)\beta^{-1}(t) - \mathcal{B}_1 - \alpha(t)\beta^{-1}(t)\mathcal{A}_2\alpha(t)\beta^{-1}(t) + \alpha(t)\beta^{-1}(t)\mathcal{B}_2 \\
            &= \mathcal{A}_1 \Phi(t) - \mathcal{B}_1 - \Phi(t)\mathcal{A}_2\Phi(t) + \Phi(t)\mathcal{B}_2,
        \end{align*}
        and the proof is achieved.
    \end{proof}

    Motivated by (\ref{optimal strategy of followers}) and (\ref{optimal strategy of leader}), we design the decentralized strategies below:
    \begin{equation}\label{decentralized strategies}
        \left\{
        \begin{aligned}
            \hat{u}_0 &= -R_0^{-1}\big[B_0\bar{y}_0 + (G - BR^{-1}L)\bar{y} - \bar{\Pi}(G - BR^{-1}L)\bar{\psi}\big], \\
            \hat{u}_i &= -R^{-1}\big[B(\bar{P}\hat{x}_i + \bar{K}\bar{x} + \bar{\phi}) + Lu_0\big], \quad i = 1,\cdots,N,
        \end{aligned}
        \right.
    \end{equation}
    where $\bar{y}_0(\cdot), \bar{y}(\cdot), \bar{\psi}(\cdot), \bar{x}(\cdot), \bar{\phi}(\cdot)$ are given by (\ref{Hamiltonian system of leader--}), and $\hat{x}_i(\cdot)$ satisfies
    \begin{align}\label{bar xi}
        d\hat{x}_i &= \big[(A - BR^{-1}B\bar{P})\hat{x}_i - BR^{-1}B(\bar{K}\bar{x} + \bar{\phi}) + Fx_0 + (G - BR^{-1}L)u_0 + f\big]dt \notag\\
        &\quad + DdW_i, \quad \hat{x}_i(0) = \xi_i,\quad i = 1,\cdots,N.
    \end{align}

    Next, we will show that the decentralized strategies (\ref{decentralized strategies}) of the leader and the followers, constitute an approximated $(\epsilon_1, \epsilon_2)$-Stackelberg equilibrium.

    \begin{mythm}\label{thm4.1}
        Assume that Assumptions (A1)-(A3) hold. Then $(\hat{u}_0(\cdot), \hat{u}_1(\cdot), \cdots, \hat{u}_N(\cdot))$ given in (\ref{decentralized strategies}) constitutes an $(\epsilon_1, \epsilon_2)$-Stackelberg equilibrium, where $\epsilon_1 = \epsilon_2 = O(\frac{1}{\sqrt{N}})$.
    \end{mythm}
    \begin{proof}{\it The followers' problem.}

        By (\ref{bar xi}) and the second equation of (\ref{Hamiltonian system of leader--}), it can be verified that
        \begin{equation}\label{estimate hat x_bar x}
            \mathbb{E} \int_0^T ||\hat{x}^{(N)} - \bar{x}||^2dt = O(\frac{1}{N}).
        \end{equation}
        For $i = 1, \cdots,N$, denote $\tilde{u}_i(\cdot) := u_i(\cdot) -\hat{u}_i(\cdot)$ and $\tilde{x}_i(\cdot) := x_i(\cdot) -\hat{x}_i(\cdot)$.
        Then $\tilde{x}_i(\cdot)$ satisfies
        \begin{equation*}
            d\tilde{x}_i = (A\tilde{x}_i + B\tilde{u}_i)dt, \quad \tilde{x}_i(0) = 0,\quad i = 1, \cdots,N.
        \end{equation*}
        Thus,
        \begin{equation*}
            \sum_{i=1}^{N} \mathbb{E} \int_0^T \left(||\tilde{x}_i(t)||^2 + ||\tilde{u}_i(t)||^2\right)dt < \infty.
        \end{equation*}
        From (\ref{cost of followers}), we have
        \begin{equation}
            J_i^N(u_i(\cdot), \hat{u}_{-i}(\cdot), u_0(\cdot)) = J_i^N(\hat{u}_i(\cdot), \hat{u}_{-i}(\cdot), u_0(\cdot)) + \tilde{J}_i^N(\tilde{u}_i(\cdot), \hat{u}_{-i}(\cdot), u_0(\cdot)) + \mathcal{I}_i^N,
        \end{equation}
        where
        \begin{equation*}
            \tilde{J}_i^N(\tilde{u}_i(\cdot), \hat{u}_{-i}(\cdot), u_0(\cdot)) := \frac{1}{2} \mathbb{E} \left\{\int_0^T \left[||\left(1 - \frac{\Gamma}{N}\right)\tilde{x}_i||^2_Q + ||\tilde{u}_i||^2_R\right]dt + ||\tilde{x}_i(T)||^2_H\right\},
        \end{equation*}
        and
        \begin{equation*}
            \mathcal{I}_i^N := \mathbb{E} \left\{\int_0^T \bigg[Q\left(1 - \frac{\Gamma}{N}\right)\tilde{x}_i\left(\hat{x}_i - \Gamma\hat{x}^{(N)} - \Gamma_1x_0 - \eta\right) + R\tilde{u}_i\hat{u}_i + \tilde{u}_iLu_0\bigg]dt + H\tilde{x}_i(T)\hat{x}_i(T)\right\}.
        \end{equation*}
        Applying It\^{o}'s formula to $\tilde{x}_i(\cdot)\big(\bar{P}(\cdot)\hat{x}_i(\cdot) + \bar{K}(\cdot)\bar{x}(\cdot) + \bar{\phi}(\cdot)\big)$, integrating from 0 to $T$, and taking expectation, we have
        \begin{equation*}
        \begin{aligned}
            &\quad \mathbb{E} \big[H\tilde{x}_i(T)\hat{x}_i(T)\big] \notag \\
            &= \mathbb{E} \big[\tilde{x}_i(T)\big(\bar{P}(T)\hat{x}_i(T) + \bar{K}(T)\bar{x}(T) + \bar{\phi}(T)\big) - \tilde{x}_i(0)\big(\bar{P}(0)\hat{x}_i(0) + \bar{K}(0)\bar{x}(0) + \bar{\phi}(0)\big)\big] \\
            &= \mathbb{E} \int_0^T \bigg\{(A\tilde{x}_i + B\tilde{u}_i)(\bar{P}\hat{x}_i + \bar{K}\bar{x} + \bar{\phi}) +\tilde{x}_i\Big[\dot{\bar{P}}\hat{x}_i + \bar{P}[(A - BR^{-1}B\bar{P})\hat{x}_i \\
            &\qquad - BR^{-1}B(\bar{K}\bar{x} + \bar{\phi}) + Fx_0 + (G - BR^{-1}L)u_0 + f\Big] +\dot{\bar{K}}\hat{x}\\
            &\qquad + \bar{K}\Big[(A - BR^{-1}B\bar{\Pi})\bar{x} - BR^{-1}B\bar{\phi} + F\bar{x}_0 + (G - BR^{-1}L)u_0 + f\Big] \\
            &\qquad -\Big[(A - \bar{\Pi}BR^{-1}B)\bar{\phi} - (G - BR^{-1}L)u_0 + \bar{\Pi}F\bar{x}_0 + \bar{\Pi}f - Q(\Gamma_1\bar{x}_0 + \eta)\Big]\bigg\}dt \\
            &= \mathbb{E} \int_0^T \bigg\{\tilde{x}_i\hat{x}_i\big[A\bar{P} + \dot{\bar{P}} + \bar{P}(A - BR^{-1}B\bar{P})\big] + \tilde{x}_i\bar{x}\big[A\bar{K} - \bar{P}BR^{-1}B\bar{K}\\
            &\qquad + \dot{\bar{K}} + \bar{K}(A - BR^{-1}B\bar{\Pi})\big] - \tilde{x}_iQ(\Gamma_1\bar{x}_0 + \eta) + B\tilde{u}_i(\bar{P}\hat{x}_i + \bar{K}\bar{x} + \bar{\phi})\bigg\}dt \\
            &= \mathbb{E} \int_0^T \Big[-\tilde{x}_iQ\big(\hat{x}_i - \Gamma \bar{x} - \Gamma_1\bar{x}_0 - \eta\big) - \tilde{u}_iR\hat{u}_i - \tilde{u}_iLu_0\Big]dt.
        \end{aligned}
        \end{equation*}
        Then,
        \begin{equation*}
        \begin{aligned}
            &\mathcal{I}_i^N = \mathbb{E} \left\{\int_0^T \left[Q\left(1 - \frac{\Gamma}{N}\right)\tilde{x}_i\left(\hat{x}_i - \Gamma\hat{x}^{(N)} - \Gamma_1x_0 - \eta\right) + R\tilde{u}_i\hat{u}_i + \tilde{u}_iLu_0\right]dt + H\tilde{x}_i(T)\hat{x}_i(T)\right\} \\
            &= \mathbb{E} \bigg\{\int_0^T \bigg[Q\left(1 - \frac{\Gamma}{N}\right)\tilde{x}_i\big(\hat{x}_i - \Gamma\bar{x} - \Gamma_1x_0 - \eta\big) + Q\left(1 - \frac{\Gamma}{N}\right)\tilde{x}_i\Gamma(\bar{x} - \hat{x}^{(N)})  \\
            &\qquad + R\tilde{u}_i\hat{u}_i + \tilde{u}_iLu_0\bigg]dt+ H\tilde{x}_i(T)\hat{x}_i(T)\bigg\}  \\
            &= -\left(1 - \frac{\Gamma}{N}\right)\mathbb{E} \left\{\int_0^T \left[\tilde{u}_iR\hat{u}_i + \tilde{u}_iLu_0\right]dt + H\tilde{x}_i(T)\hat{x}_i(T)\right\}  \\
            &\quad + \mathbb{E} \left\{\int_0^T \left[Q\left(1 - \frac{\Gamma}{N}\right)\tilde{x}_i\Gamma(\bar{x} - \hat{x}^{(N)}) + R\tilde{u}_i\hat{u}_i + \tilde{u}_iLu_0\right]dt + H\tilde{x}_i(T)\hat{x}_i(T)\right\}  \\
            &=\mathbb{E} \left\{\int_0^T \left[\frac{\Gamma}{N}\tilde{u}_iR\hat{u}_i + \frac{\Gamma}{N}\tilde{u}_iLu_0 + Q\left(1 - \frac{\Gamma}{N}\right)\tilde{x}_i\Gamma(\bar{x} - \hat{x}^{(N)})\right]dt + \frac{\Gamma}{N}H\tilde{x}_i(T)\hat{x}_i(T)\right\}  =O(\frac{1}{\sqrt{N}}).
        \end{aligned}
        \end{equation*}
        Thereby,
        $$J_i^N(\hat{u}_i(\cdot), \hat{u}_{-i}(\cdot), u_0(\cdot)) \le J_i^N(u_i(\cdot), \hat{u}_{-i}(\cdot), u_0(\cdot)) + \epsilon_1.$$
        Thus, $(\hat{u}_1(\cdot),\cdots,\hat{u}_N(\cdot))$ is an $\epsilon_1$-Nash equilibrium, where $\epsilon_1 = O(\frac{1}{\sqrt{N}})$.

        {\it Problem of the leader.}

        Let $\check{X} := [\check{x}_0, \check{y}^{(N)}, \check{\phi}]^\top$, $\check{Y} := [\check{y}_0, \check{x}^{(N)}, \check{\psi}]^\top$, $\check{Z} := [\check{z}_0, 0, \check{V}_N]^\top$, $\check{\mathcal{D}} := [-C_0\check{y}_0, 0, 0]^\top$, $\check{Z}^{(N)} := [0, \check{\beta}^{(N)}, 0]^\top$, $\check{D}^{(N)} := [0, D, 0]^\top$ and
        \begin{equation*}
        \begin{aligned}
            \check{\mathcal{A}}_1 &:=
            \begin{bmatrix}
                A_0 & - B_0R_0^{-1}(G - BR^{-1}L) & 0 \\
                Q_0\Gamma_0 & - A + BR^{-1}B\Pi_N & 0 \\
                - \Pi_NF + (1 - \frac{\Gamma}{N})Q\Gamma_1 & \Pi_NR_0^{-1}(G - BR^{-1}L)^2 & - A + BR^{-1}B\Pi_N \\
            \end{bmatrix},\\
            \check{\mathcal{B}}_1 &:=
            \begin{bmatrix}
                B_0R_0^{-1}B_0 & 0 & - B_0R_0^{-1}\Pi_N(G - BR^{-1}L) \\
                0 & Q_0\Gamma_0^2 & 0 \\
                - B_0R_0^{-1}\Pi_N(G - BR^{-1}L) & 0 & \Pi_N^2R_0^{-1}(G - BR^{-1}L)^2 \\
            \end{bmatrix},\\
            \check{\mathcal{A}}_2 &:=
            \begin{bmatrix}
                -Q_0 & -F & 0 \\
                F & -R_0^{-1}(G - BR^{-1}L)^2 & -BR^{-1}B \\
                0 & BR^{-1}B & 0 \\
            \end{bmatrix},\\
            \check{\mathcal{B}}_2 &:=
            \begin{bmatrix}
                A_0 & - Q_0\Gamma_0 & - \Pi_NF + (1 - \frac{\Gamma}{N})Q\Gamma_1 \\
                B_0R_0^{-1}(G - BR^{-1}L) & - A + BR^{-1}B\Pi_N & -\Pi_NR_0^{-1}(G - BR^{-1}L)^2 \\
                0 & 0 & - A + BR^{-1}B\Pi_N \\
            \end{bmatrix},\\
            \check{\mathcal{C}} &:=
            \begin{bmatrix}
                C_0 & 0 & 0 \\
                0 & 0 & 0 \\
                0 & 0 & 0 \\
            \end{bmatrix},\quad
            \check{\mathfrak{f}}_0 :=
            \begin{bmatrix}
                f_0 \\
                - Q_0\Gamma_0\eta_0 \\
                (1 - \frac{\Gamma}{N})Q\eta - \Pi_Nf \\
            \end{bmatrix},\quad
            \check{\mathfrak{f}} :=
            \begin{bmatrix}
                Q_0\eta_0 \\
                f \\
                0 \\
            \end{bmatrix}.
        \end{aligned}
        \end{equation*}
        With the above notions, we can rewrite (\ref{Hamiltonian system of leader}) as
        \begin{equation}\label{Hamiltonian system of leader-}
            \left\{
            \begin{aligned}
                d\check{X} &= \big[\check{\mathcal{A}}_1\check{X} - \check{\mathcal{B}}_1\check{Y} + \check{\mathcal{C}}\check{Z} + \check{\mathfrak{f}}_0\big]dt + \check{Z}dW_0 + \check{Z}^{(N)}dW^{(N)}, \quad \check{X}(T) = [\xi_0, 0, 0]^\top, \\
                d\check{Y} &= \big[\check{\mathcal{A}}_2\check{X} - \check{\mathcal{B}}_2\check{Y} + \check{\mathfrak{f}}\big]dt + \check{\mathcal{D}}dW_0 + \check{\mathcal{D}}^{(N)}dW^{(N)}, \quad \check{Y}(0) = [H_0\check{x}_0(0), \xi^{(N)}, 0]^\top.
            \end{aligned}
            \right.
        \end{equation}

        Denote $\tilde{X} := \check{X} - X$, $\tilde{Y} := \check{Y} - Y$, $\tilde{Z} := \check{Z} - Z$, $\tilde{\mathcal{D}} := \check{\mathcal{D}} - \mathcal{D}$, and $\tilde{\mathcal{A}}_1 := \check{\mathcal{A}}_1 - \mathcal{A}_1$, $\tilde{\mathcal{B}}_1 := \check{\mathcal{B}}_1 - \mathcal{B}_1$, $\tilde{\mathcal{B}}_2 := \check{\mathcal{B}}_2 - \mathcal{B}_2$, $\tilde{\mathfrak{f}}_0 := \check{\mathfrak{f}}_0 - \mathfrak{f}_0$, we have
        \begin{equation}\label{Hamiltonian system of leader----}
            \left\{
            \begin{aligned}
                d\tilde{X} &= \big[\mathcal{A}_1\tilde{X} + \tilde{\mathcal{A}}_1\check{X} - \mathcal{B}_1\tilde{Y} - \tilde{B}_1\check{Y} + \mathcal{C}\tilde{Z} + \tilde{\mathfrak{f}}_0\big]dt + \tilde{Z}dW_0 + \check{Z}^{(N)}dW^{(N)}, \\
                d\tilde{Y} &= \big[\mathcal{A}_2\tilde{X} - \mathcal{B}_2\tilde{Y} - \tilde{\mathcal{B}}_2\check{Y}\big]dt + \tilde{\mathcal{D}}dW_0 + \check{\mathcal{D}}^{(N)}dW^{(N)}, \\
                \quad \tilde{X}(T) &= [\xi_0, 0, 0]^\top,\quad \tilde{Y}(0) = \big[H_0(\check{x}_0(0) - \bar{x}_0(0)), \xi^{(N)} - \bar{\xi}, 0\big]^\top.
            \end{aligned}
            \right.
        \end{equation}

        We denote $\varpi := \tilde{\mathcal{A}}_1\check{X} - \tilde{B}_1\check{Y} + \tilde{\mathfrak{f}}_0$, $\rho := - \tilde{\mathcal{B}}_2\check{Y}$. By the continuous dependence of the solution on the parameter in Theorem 4 of \cite{Huang-Zhou-20}, we have $\sup_{0 \le t \le T}||\varpi||^2 = O(\frac{1}{N})$, $\sup_{0 \le t \le T}||\rho||^2 = O(\frac{1}{N})$.
        Then we can rewrite (\ref{Hamiltonian system of leader----}) as
        \begin{equation}\label{Hamiltonian system of leader--------}
            \left\{
            \begin{aligned}
                d\tilde{X} &= \big[\mathcal{A}_1\tilde{X} - \mathcal{B}_1\tilde{Y} + \mathcal{C}\tilde{Z} + \varpi\big]dt + \tilde{Z}dW_0 + \check{Z}^{(N)}dW^{(N)},\\
                d\tilde{Y} &= \big[\mathcal{A}_2\tilde{X} - \mathcal{B}_2\tilde{Y} + \rho\big]dt + \tilde{\mathcal{D}}dW_0 + \check{\mathcal{D}}^{(N)}dW^{(N)},\\
                \tilde{X}(T) &= [\xi_0, 0, 0]^\top,\quad \tilde{Y}(0) = \big[H_0(\check{x}_0(0) - \bar{x}_0(0)), \xi^{(N)} - \bar{\xi}, 0\big]^\top.
            \end{aligned}
            \right.
        \end{equation}

        We assume that $\tilde{X}(\cdot) = \Theta(\cdot)\tilde{Y}(\cdot)$, where $\Theta(\cdot)$ is a differential function with $\Theta(T) = 0$.
        By It\^{o}'s formula, we have
        \begin{equation*}
        \begin{aligned}
            d\tilde{X} &= \dot{\Theta}\tilde{Y}dt + \Theta\big[(\mathcal{A}_2\Theta\tilde{Y} - \mathcal{B}_2\tilde{Y} + \rho)dt + \tilde{\mathcal{D}}dW_0 + \check{\mathcal{D}}^{(N)}dW^{(N)}\big] \\
            &= \big[\mathcal{A}_1\Theta\tilde{Y} - \mathcal{B}_1\tilde{Y} + \mathcal{C}\tilde{Z} + \varpi\big]dt + \tilde{Z}dW_0 + \check{Z}^{(N)}dW^{(N)}.
        \end{aligned}
        \end{equation*}
        Comparing the corresponding coefficients, we have
        \begin{equation}\label{Riccati of leader---}
            \dot{\Theta} + \Theta\mathcal{A}_2\Theta - \Theta\mathcal{B}_2 - \mathcal{A}_1\Theta + \mathcal{B}_1 = 0, \quad \Theta(T) = 0.
        \end{equation}
        Then we achieve
        \begin{equation*}
            d\tilde{Y} = \big[(\mathcal{A}_2\Theta - \mathcal{B}_2)\tilde{Y} + \rho\big]dt + \tilde{\mathcal{D}}dW_0 + \check{\mathcal{D}}^{(N)}dW^{(N)},
        \end{equation*}
        with $\tilde{Y}(0) = \big[H_0(\check{x}_0(0) - \bar{x}_0(0)), \xi^{(N)} - \bar{\xi}, 0\big]^\top$.
        This implies
        \begin{equation*}
            \tilde{Y}(t) = e^{(\mathcal{A}_2\Theta - \mathcal{B}_2)t}\tilde{Y}(0) + \int_{0}^{t} e^{(\mathcal{A}_2\Theta - \mathcal{B}_2)(t-s)}\big[\rho ds + \tilde{\mathcal{D}}dW_0(s) + \check{\mathcal{D}}^{(N)}dW^{(N)}(s)\big],
        \end{equation*}
        which gives
        \begin{equation}\label{estimate tilde Y}
            \sup_{0 \le t \le T} \mathbb{E} ||\tilde{Y}(t)||^2 = O(\frac{1}{N}),
        \end{equation}
        and since $\tilde{X}(\cdot) = \Theta(\cdot)\tilde{Y}(\cdot)$, then
        \begin{equation}\label{estimate tilde X}
            \sup_{0 \le t \le T} \mathbb{E} ||\tilde{X}(t)||^2 = O(\frac{1}{N}),
        \end{equation}
        which means
        \begin{equation}\label{estimates-1}
            \mathbb{E} \int_0^T ||\check{x}_0 - \bar{x}_0||^2dt = O(\frac{1}{N}), \quad \mathbb{E} \int_0^T ||\check{x}^{(N)} - \bar{x}||^2dt = O(\frac{1}{N}),
        \end{equation}
        and
        \begin{equation}\label{estimates-2}
        \begin{aligned}
        \mathbb{E} \int_0^T ||\check{y}_0 - \bar{y}_0||^2dt = O(\frac{1}{N}),\quad \mathbb{E} \int_0^T ||\check{y}^{(N)} - \bar{y}||^2dt = O(\frac{1}{N}), \quad \mathbb{E} \int_0^T ||\hat{\psi} - \bar{\psi}||^2dt = O(\frac{1}{N}).
        \end{aligned}
        \end{equation}
        Combined with (\ref{estimate hat x_bar x}), we have
        \begin{equation}\label{estimate hat x_check x}
            \mathbb{E} \int_0^T ||\hat{x}^{(N)} - \check{x}^{(N)}||^2dt = O(\frac{1}{N}),
        \end{equation}
        and $\bar{x}_0(t)= \hat{x}_0(t)$, for all $t\in[0,T]$. Then
        \begin{equation}\label{estimate hat x_check x 0}
            \mathbb{E} \int_0^T ||\hat{x}_0 - \check{x}_0||^2dt = O(\frac{1}{N}).
        \end{equation}

        We need two steps. First, by (\ref{estimate hat x_check x}), we get
        \begin{equation}\label{estimates-cost}
        \begin{aligned}
            &J_0^N(\check{u}_0(\cdot), \check{u}^N[\cdot;\check{u}_0(\cdot)]) \le J_0^N(u_0(\cdot), \check{u}^N[\cdot;u_0(\cdot)])\\
            &= \frac{1}{2} \mathbb{E} \left\{\int_0^T \left[||x_0 - \Gamma_0\check{x}^{(N)} - \eta_0||^2_{Q_0} + ||u_0||^2_{R_0}\right]dt + ||x_0(0)||^2_{H_0}\right\} \\
            &= \frac{1}{2} \mathbb{E} \left\{\int_0^T \left[||x_0 - \Gamma_0\hat{x}^{(N)} - \eta_0 + \Gamma_0\left(\hat{x}^{(N)} - \check{x}^{(N)}\right)||^2_{Q_0} + ||u_0||^2_{R_0}\right]dt + ||x_0(0)||^2_{H_0}\right\} \\
            &\le J_0^N(u_0(\cdot), \hat{u}^N[\cdot;u_0(\cdot)]) + O(\frac{1}{N}) + O(\frac{1}{\sqrt{N}}).
        \end{aligned}
        \end{equation}
        Next, from (\ref{estimates-2})-(\ref{estimate hat x_check x 0}), we have
       \begin{equation}\label{estimates-cost-}
        \begin{aligned}
            &J_0^N(\hat{u}_0(\cdot), \hat{u}^N[\cdot;\hat{u}_0(\cdot)]) \\
            &= \frac{1}{2} \mathbb{E} \left\{\int_0^T \left[||\hat{x}_0 - \Gamma_0\hat{x}^{(N)} - \eta_0||^2_{Q_0} + ||\hat{u}_0||^2_{R_0}\right]dt + ||\hat{x}_0(0)||^2_{H_0}\right\} \\
            &= \frac{1}{2} \mathbb{E} \bigg\{\int_0^T \Big[||\hat{x}_0 - \Gamma_0\hat{x}^{(N)} - \eta_0||^2_{Q_0} + ||R_0^{-1}\big[B_0\bar{y}_0 + (G - BR^{-1}L)\bar{y}  \\
            &\qquad\quad - \bar{\Pi}(G - BR^{-1}L)\bar{\psi}\big]||^2_{R_0}\Big]dt+ ||\hat{x}_0(0)||^2_{H_0}\bigg\}  \\
            &= \frac{1}{2} \mathbb{E} \bigg\{\int_0^T \Big[||\check{x}_0 - \Gamma_0\check{x}^{(N)} - \eta_0 + (\hat{x}_0 - \check{x}_0) + \Gamma_0\left(\check{x}^{(N)} - \hat{x}^{(N)}\right)||^2_{Q_0}  \\
            &\qquad\quad + ||R_0^{-1}\big[B_0(\check{y}_0 + (\bar{y}_0 - \check{y}_0))+ (G - BR^{-1}L)\left(\check{y}^{(N)} + (\bar{y} - \check{y}^{(N)})\right)  \\
            &\qquad\quad - \bar{\Pi}(G - BR^{-1}L)(\check{\psi} + (\bar{\psi} - \check{\psi}))\big]||^2_{R_0}\Big]dt + ||\check{x}_0(0) + (\hat{x}_0(0) - \check{x}_0(0))||^2_{H_0}\bigg\}  \\
            &\le J_0^N(\check{u}_0(\cdot), \check{u}^N[\cdot;\check{u}_0(\cdot)]) + O(\frac{1}{N}) + O(\frac{1}{\sqrt{N}}).
        \end{aligned}
        \end{equation}
        Hence, from (\ref{estimates-cost}) and (\ref{estimates-cost-}), we have
        $$J_0^N(\hat{u}_0(\cdot), \hat{u}^N[\cdot;\hat{u}_0(\cdot)]) \le J_0^N(u_0(\cdot), \hat{u}^N[\cdot;u_0(\cdot)]) + \epsilon_2.$$
        Thus, by Definition \ref{def2.1}, $(\hat{u}_0(\cdot), \hat{u}_1(\cdot),\cdots,\hat{u}_N(\cdot))$ is an $(\epsilon_1, \epsilon_2)$-Stackelberg equilibrium, where $\epsilon_1 = \epsilon_2 = O(\frac{1}{\sqrt{N}})$. The proof is complete.
    \end{proof}

    \section{Numerical simulation}

    This section provides numerical examples for Problem {\bf (P0)} to verify our results.
    We set $A_0=B_0=C_0=f_0=A=B=F=G=f=G=Q_0=R_0=\eta_0=Q=R=H=\eta=L=1$, $D=0.05$, $H_0=0$, $\Gamma=\Gamma_0=\Gamma_1=0.5$. The time interval is $[0,5]$.
    The initial states of followers are independently drawn from a uniform distribution on $[0,10]$, and the terminal condition of the leader is $\xi_0 \sim N(0,5)$.
    The curve of $\bar{P}(t)$, $\bar{K}(t)$ and $\bar{\Pi}(t)$, described by (\ref{Riccati-1 of followers--})-(\ref{Riccati-1+2 of followers--}), are shown in Figure \ref{PKPi_plot}.
    \begin{figure}[htbp]
        \centering\includegraphics[width=12cm]{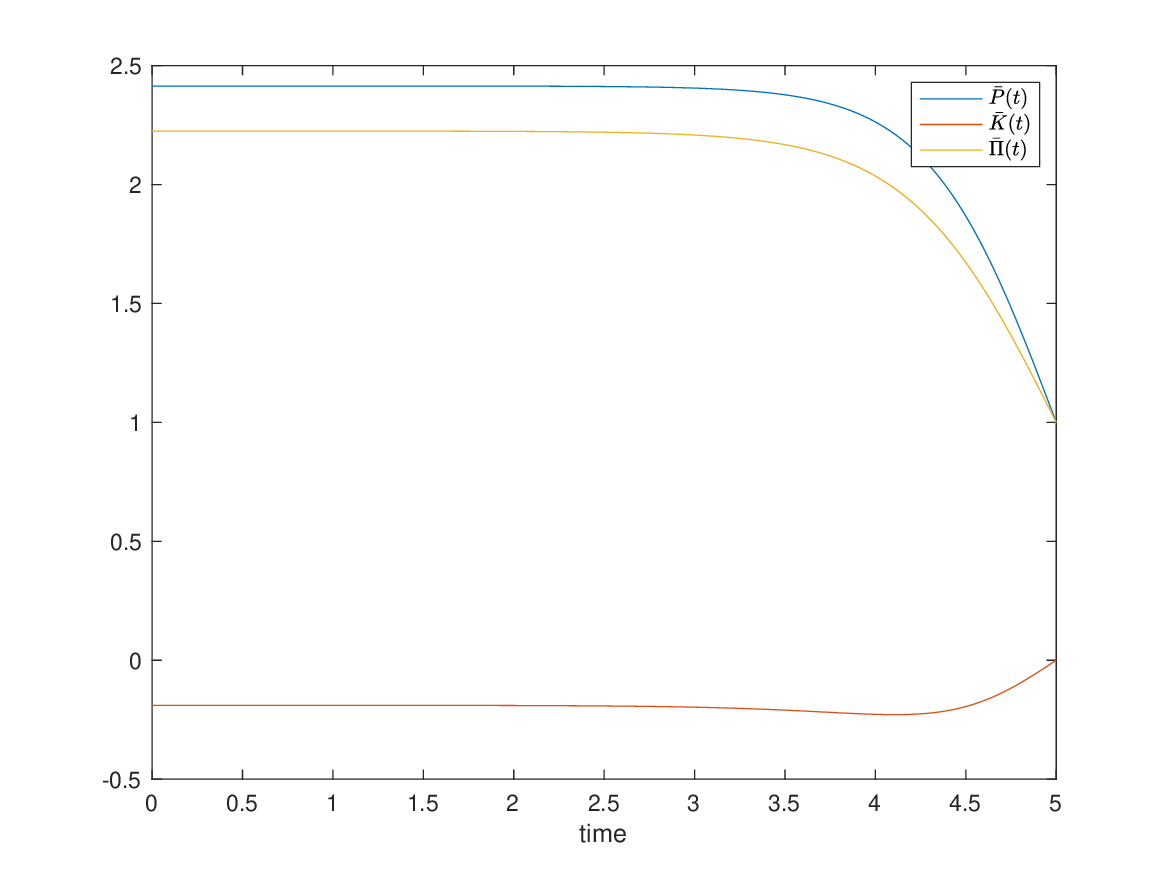}
        \caption{The curves of $\bar{P}(t)$, $\bar{K}(t)$ and $\bar{\Pi}(t)$}
        \label{PKPi_plot}
    \end{figure}
    
    The solution to the Riccati equation (\ref{Riccati of leader}), $\Phi(\cdot) \in \mathbb{R}^{3 \times 3}$, is depicted in Figure \ref{Phi_plot}. Note that
    $\Phi(\cdot)$ is neither symmetric nor positive semi-definite.
    \begin{figure}[htbp]
        \centering\includegraphics[width=12cm]{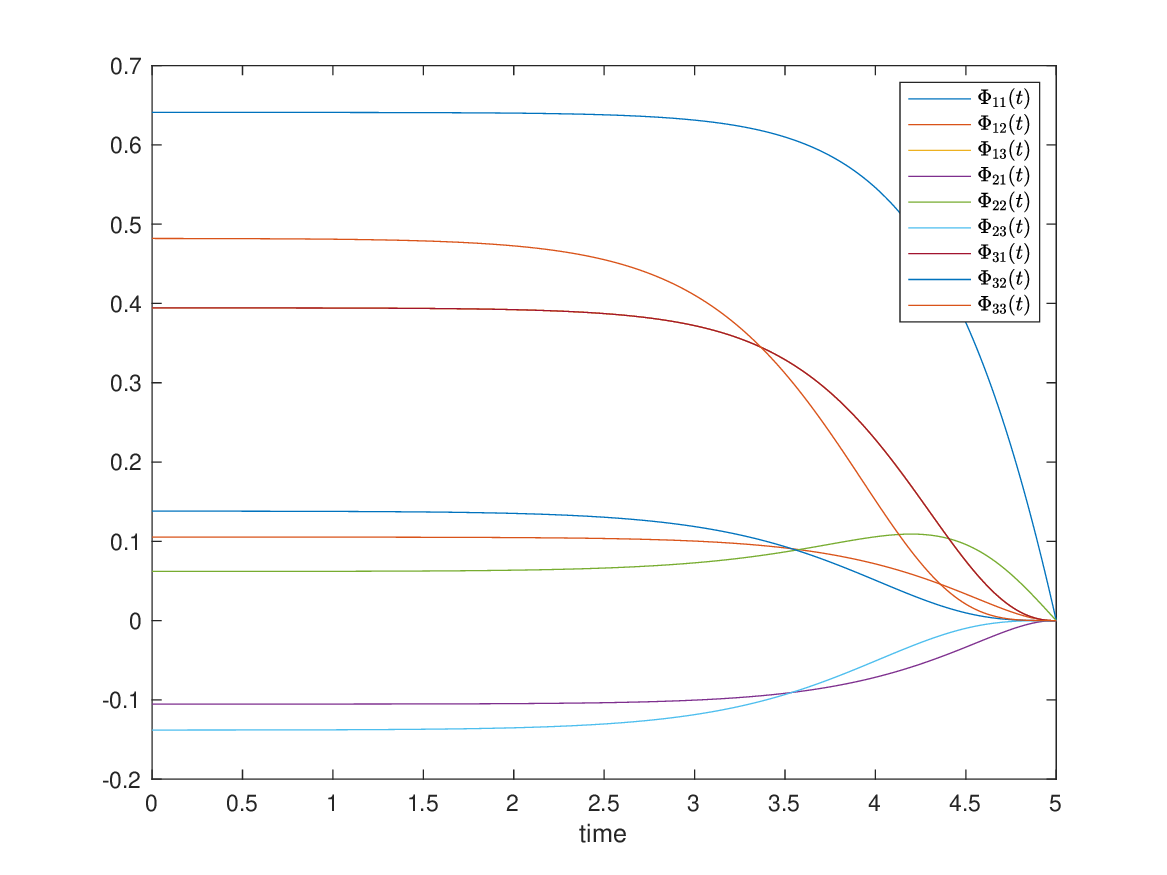}
        \caption{The curve of $\Phi(t)$}
        \label{Phi_plot}
    \end{figure}
    
    We denote $\epsilon(N) = \bigl(\mathbb{E}\int_{0}^{T}||\hat{x}^{(N)} - \bar{x}||^2 dt\bigr)^{\frac{1}{2}}$ to depict the performance of the decentralized strategies.
    The graph of $\epsilon(N)$ with respect to $N$ is shown in Figure \ref{epsilon_plot}, which confirms the consistency of the mean field approximation.
    \begin{figure}[htbp]
        \centering\includegraphics[width=12cm]{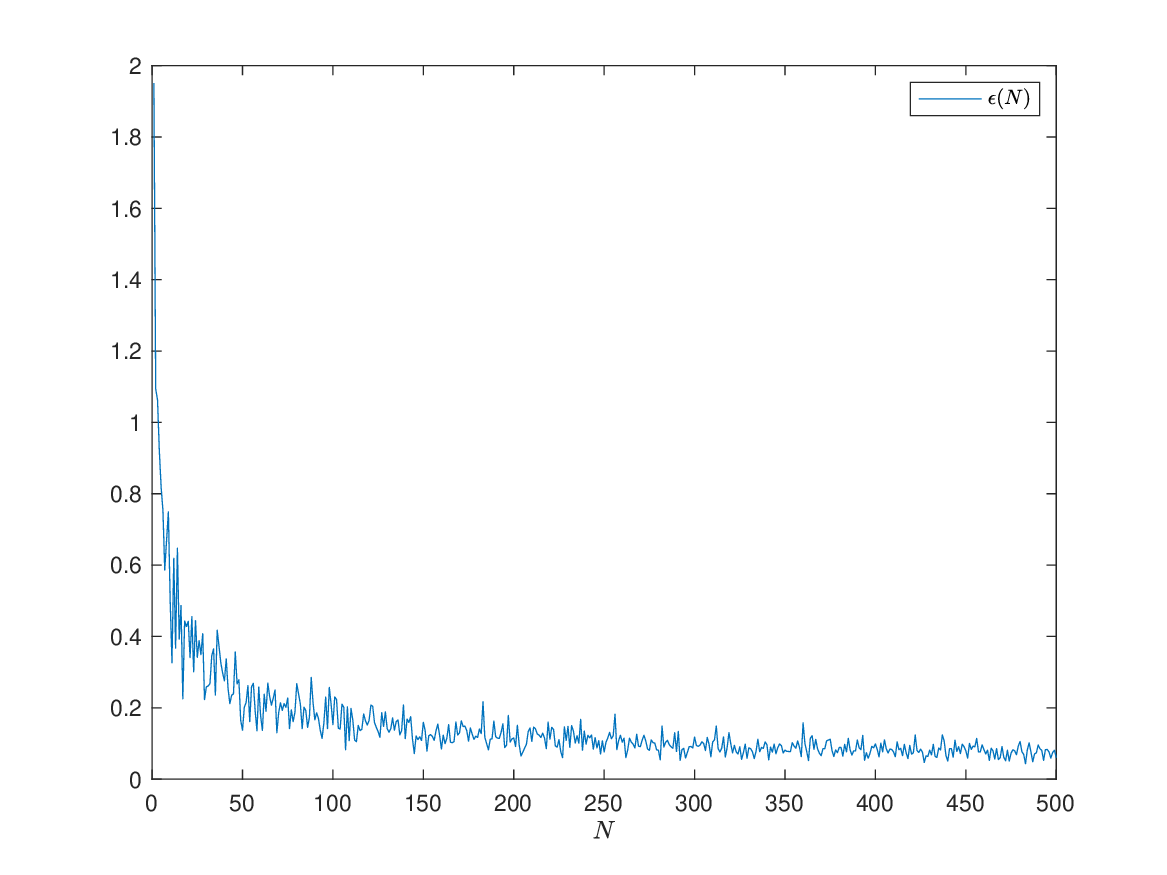}
        \caption{$\epsilon(N)$ with respect to $N$}
        \label{epsilon_plot}
    \end{figure}

    \section{Conclusions}

    In this paper, we have explored a backward-forward mean field LQ Stackelberg stochastic differential game featuring a backward leader and numerous forward followers.
    Initially, we address a mean field game problem concerning $N$ followers. Through the decoupling of the high-dimensional Hamiltonian system using mean field approximations, we formulate a set of open-loop decentralized strategies for all players.
    Subsequently, these strategies are demonstrated to constitute an $(\epsilon_1, \epsilon_2)$-Stackelberg equilibrium.

    Other topics including the case when the diffusion coefficients are dependent of state and control variables (\cite{Xu-Zhang-20}, \cite{Si-Wu-21}), and with common noise (\cite{Bensoussan-Feng-Huang-21}), are our future research interests.

\end{document}